\documentclass[11pt]{article}

\usepackage{amsmath}
\usepackage{amsthm}
\usepackage{amssymb}
\usepackage{graphicx}
\usepackage{comment}

\usepackage{latexsym,array,amsfonts,eucal}
\usepackage{mathrsfs,dsfont,bbm}
\usepackage{titlesec,fancyhdr,titletoc}
\usepackage{verbatim}

\begin{document}

\title{Precise asymptotics for large deviations of integral forms}
\author{Xiangfeng Yang\thanks{xyang2@tulane.edu}\\ Grupo de F\'{i}sica Matem\'{a}tica\\Universidade de Lisboa\\ Av. Prof. Gama Pinto 2, 1649-003 Lisboa, Portugal}

\date{September 13, 2012}

\maketitle

\begin{abstract}
For suitable families of locally infinitely
divisible Markov processes $\{\xi^{{\epsilon}}_t\}_{0\leq t\leq T}$ with frequent small jumps depending on a small parameter $\epsilon>0,$ precise asymptotics for large deviations of integral forms $\mathbb{E}^{\epsilon}\left[\exp\{{\epsilon}^{-1}F(\xi^{\epsilon})\}\right]$ are proved for smooth functionals $F.$ The main ingredient of the proof in this paper is a recent result regarding the asymptotic expansions of the expectations $\mathbb{E}^{\epsilon}\left[G(\xi^{\epsilon})\}\right]$ for smooth $G.$ Several connections between these large deviation asymptotics and partial integro-differential equations are included as well.
\end{abstract}

\renewcommand{\theequation}{\thesection.\arabic{equation}}
\newtheorem{theorem}{Theorem}[section]
\newtheorem{thm}{Theorem}[section]
\newtheorem{cor}[thm]{Corollary}
\newtheorem{lemma}[thm]{Lemma}
\newtheorem{prop}[thm]{Proposition}
\theoremstyle{definition}
\newtheorem{defn}[thm]{Definition}
\theoremstyle{remark}
\newtheorem{remark}[thm]{Remark}
\numberwithin{equation}{section}
\newcommand{\dx}{\Delta x}
\newcommand{\dt}{\Delta t}

\def\keywords{\vspace{.5em}\hspace{-2em}
{\textit{Keywords and phrases}:\,\relax%
}}
\def\endkeywords{\par}

\def\MSC{\vspace{.0em}\hspace{-2em}
{\textit{AMS 2010 subject classifications}:\,\relax%
}}
\def\endMSC{\par}

\keywords{Cr\'{a}mer's transformation, large deviations, normal deviations, locally infinitely
divisible, compensating operators}

\MSC{Primary 60F10, 60F17; secondary 60J75, 35C20}

\section{Introduction}\label{sec:introduction}
The study of large deviations in limit theorems can be formulated as
follows. Let $\mathbb{X}$ be a metric space with metric $\rho$, and
$\mu^{\epsilon}$ be a family of probability measures on $\mathbb{X}$
depending on a parameter $\epsilon>0.$ Suppose there is a point
$x_{*}\in\mathbb{X}$ such that for any $\delta>0$ and small
$\epsilon,$ $\mu^{\epsilon}\{y:\rho(x_{*},y)<\delta\}$ have
overwhelming probabilities:
$\lim_{\epsilon\rightarrow0}\mu^{\epsilon}\{y:\rho(x_{*},y)<\delta\}=1.$
Problems on large deviations are concerned with the limiting
behavior as $\epsilon\rightarrow0$ of the infinitesimal
probabilities $\mu^{\epsilon}(A)$ for measurable sets
$A\subseteq\mathbb{X}$ that are situated at a positive distance from
point $x_{*}.$ Problems concerning asymptotics as
$\epsilon\rightarrow0$ of integrals in the form
$\int_{\mathbb{X}}f_{\epsilon}(x)\mu^{\epsilon}(dx)$ also belong to
large deviations if the main part of such integrals for small
$\epsilon$ is due to the values of $x$ far away from point $x_{*}.$ This
paper deals with the later (asymptotics of integrals).

Classical large deviation problems are about empirical means $\bar{S}_n=\sum_{i=1}^n\xi_i/n$ of random variables $\xi_i.$ In general, results obtained deal with
asymtotics up to logarithmic equivalence $\ln \mathbb{P}\{\bar{S}_n\in A\}$ or $\ln \mathbb{E}\exp\{nf(\bar{S}_n)\},$ and we call these results \emph{rough large deviations}, see \cite{Wentzell-LD-Markov-Processes-1986}, \cite{Dembo-Zei}, \cite{Dupuis-Ellis-1997} and \cite{Yang-SPL-2012}. If we assume that the random variables $\xi_i$ are real-valued and independent identically distributed (i.i.d.), Cram\'{e}r in \cite{Cramer-1938} made use of limit theorems on normal
deviations (asymptotic expansions in limit theorem for i.i.d. random
variables) and proved a precise large deviation result:
$\mathbb{P}\{\bar{S}_n>a\}\sim\frac{c}{\sqrt{n}}\exp\{-nI(a)\}$
for $a>0,$ some constant $c$ and a rate function $I(x)$ provided $\xi_i$ are non-lattice having zero mean and finite moment generating
function. He used what we will call Cram\'{e}r's transformation to
define a new distribution $\tilde{\mu}(dx)=e^{z_0x}\mu(dx)/\int
e^{z_0x}\mu(dx)$ for some $z_0$ so that new random variables
$\tilde{\xi}_i$ corresponding to $\tilde{\mu}$ have mean $a.$ If
more conditions are assumed on $\xi_i,$ then Cram\'{e}r derived precise asymptotics for large deviation probabilities
$\mathbb{P}\{\bar{S}_n>a\}=\exp\{-nI(a)\}(\sum_{1\leq i\leq
N}l_in^{-i/2}+o(n^{-N/2}))$ for an integer $N$ depending on the moments of $\xi_i$ (see also \cite{Bhattacharya-Rao-1976}, \cite{Iltis-1995} and the references therein for related works). If we use $\mu^n$
to denote the distributions of $\bar{S}_n,$ then
results concerning integrals $\int_{\mathbb{R}}f_n(x)\mu^n(dx)$ with
$f_n(x)=\exp\{nf(x)\}$ can be obtained similarly in the form
$\int_{\mathbb{R}}\exp\{nf(x)\}\mu^n(dx)=\exp\{n[f(x_0)-I(x_0)]\}(\sum_{0\leq
i\leq M}k_in^{-i}+o(n^{-M}))$ provided $\max[f(x)-I(x)]$ is reached
uniquely at $x_0$ for some integer $M$ depending on the smoothness of $f(x).$ If $\xi_i$ are not independent or the moment generating
function doesn't exist, similar precise large deviations can be also obtained (see for instance \cite{Nagaev-1969}, \cite{Nagaev-1979}, \cite{Mikosch-Wintenberger-2012} and the references therein). For related treatments on other types of sequences of random variables (such as randomly indexed sums), we refer to \cite{Mikosch-Nagaev-1998} and \cite{Ng-Tang-Yan-Yang-2004}. Precise large deviations are also called in the literature as sharp (or exact) large deviations.

When we study large deviations for stochastic processes
$\xi_t^{\epsilon}$ defined on probability spaces
$(\Omega,\mathcal{F},\mathbb{P}^{\epsilon}),$ one usually investigates the asymtotics up to logarithmic equivalence $\ln
\mathbb{P}^{\epsilon}\{\xi^{\epsilon}\in A\}\sim g(\epsilon,A)$ or $\ln
\mathbb{E}^{\epsilon}f_{\epsilon}(\xi^{\epsilon})\sim
g(\epsilon,f_{\epsilon})$ (we use $\mathbb{E}^{\epsilon}$ to denote the expectation with respect to the probability measure $\mathbb{P}^{\epsilon}$). More precisely, if $(\mathbb{X},\mathcal{B})$ denotes a function space with a metric and the Borel $\sigma$-algebra $\mathcal{B},$ then the family $\{\xi^{\epsilon}\}$ is said to satisfy the \textit{large deviation principle} with a normalized action functional $S(x)$ on $(\mathbb{X},\mathcal{B})$ if for every Borel measurable set $\Gamma\in\mathcal{B},$
\begin{equation}\label{eq:LDP-definition}
-\inf_{x\in\Gamma^o}S(x)\leq \liminf_{\epsilon\rightarrow0}\epsilon\ln \mathbb{P}^{\epsilon}\{\xi^\epsilon\in \Gamma\}\leq \limsup_{\epsilon\rightarrow0}\epsilon\ln \mathbb{P}^{\epsilon}\{\xi^\epsilon\in \Gamma\}\leq -\inf_{x\in\bar{\Gamma}}S(x)
\end{equation}
where $S(x)$ takes values in $[0,+\infty]$ such that each level set $\Phi(s):=\{x\in \mathbb{X}: S(x)\leq s\}$ is compact ($s\geq0$). The normalized action functional $S(x)$ is also called a rate function in the literature. Here we also consider the large deviation principle as rough large deviations. We refer to \cite{Dembo-Zei}, \cite{Wentzell-LD-Markov-Processes-1986}, \cite{Varadhan-1969}, \cite{Mogulskii-1993}, \cite{Feng-Kurtz-2006}, \cite{Acosta-1994} and \cite{Knessl-1985} for the large deviation principles for various classes of stochastic processes. Of note, references \cite{Acosta-1994} and \cite{Knessl-1985} study processes with jumps, which will be included in this paper.

The following identity, to be called as \textit{Varadhan's integral lemma} according to \cite{Dembo-Zei}, was derived in \cite{Varadhan-1969} from (\ref{eq:LDP-definition})
\begin{equation}\label{eq:Varadhan-integral-lemma}
\lim_{\epsilon\rightarrow0}\epsilon\ln \mathbb{E}^{\epsilon}\exp\left\{\epsilon^{-1}F(\xi^{\epsilon})\right\}=\max_{x\in\mathbb{X}}[F(x)-S(x)]
\end{equation}
for every bounded and continuous functional $F(x)$ on $\mathbb{X}.$ If the metric space $\mathbb{X}$ is regular enough, then (\ref{eq:Varadhan-integral-lemma}) and (\ref{eq:LDP-definition}) are equivalent, see Section 3.3 in \cite{Freidlin-Wentzell-1998}, \cite{Bryc-1990} and \cite{Dembo-Zei}. Related works were considered in \cite{Donsker-Varadhan-I-IV-1975-1983}. Two questions arise here. First, it is natural to expect precise large deviation probabilities from (\ref{eq:LDP-definition}) for suitable stochastic processes. This direction has been extensively studied, such as for random walks, actual aggregate loss processes, prospective-loss processes, (fractional) Ornstein-Uhlenbeck processes, Gaussian quadratic forms, Markov chains and so on (see \cite{Shen-Lin-Zhang-2009}, \cite{Ng-Tang-Yan-Yang-2003}, \cite{Pisztora-Povel-Zeitouni-1999}, \cite{Arous-Deuschel-Stroock-1993}, \cite{Bercu-Gamboa-Lavielle-2000}, \cite{Bercu-Coutin-Savy-2011}, \cite{Bercu-Coutin-Savy-2012}, \cite{Fatalov-2006} and \cite{Fatalov-2011}). Second, it is natural to expect precise large deviations of integral forms from (\ref{eq:Varadhan-integral-lemma}) for more regular $F$ such as what
we had for sums of i.i.d. random variables. Namely, we want to specify the conditions on $F$ and $\xi^{\epsilon}$ under which $\mathbb{E}^{\epsilon}\exp\left\{\epsilon^{-1}F(\xi^{\epsilon})\right\}$ has precise asymptotics. Not many references can be found along this direction, and below is a summary.

Indeed, for the family of stochastic processes $\{\sqrt{\epsilon}W_t\}_{t\in[0,T]},$ where
$\{W_t\}$ is the standard Wiener process, it was proved by Schilder in \cite{Schilder} that the following precise asymptotics hold
\begin{equation}\label{eq:PLD-Wiener-processes}
\mathbb{E}^{\epsilon}\left[\exp\left\{{\epsilon}^{-1}F(\sqrt{\epsilon}W)\right\}\right]=\exp\{{\epsilon}^{-1}[F(\phi_0)-S(\phi_0)]\}\left(\sum_{0\leq
i\leq s/2}K_i\cdot
{\epsilon}^i+o({\epsilon}^{s/2})\right)
\end{equation}
for a positive integer $s$ depending on the smoothness of $F,$ where the normalized action functional
$S(\phi)=\frac{1}{2}\int_0^T\phi'(t)^2dt$ for absolutely continuous
$\phi$ and $S(\phi)=\infty$ for other $\phi.$ We note that the trajectory metric space $\mathbb{X}$ here is the continuous function space $C[0,T],$ and $F(\phi_0)-S(\phi_0)$ indicates the maximum of $F(\phi)-S(\phi)$ is reached uniquely at $\phi_0.$ The proof of (\ref{eq:PLD-Wiener-processes}) made use of many particular properties of Wiener processes such as $d\mu_{\phi+\sqrt{\epsilon}W}/d\mu_{W}$ and the distribution of $\max_{t\in[0,T]}|W_t|.$ In 1970s, Dubrovskii and Wentzell showed precise large deviation of the first order, i.e. $s=0$ in (\ref{eq:PLD-Wiener-processes}), for suitable Markov processes by a transformation similar to Cram\'{e}r's. The general precise asymptotics can't be obtained due to the lack of tools which will be explained below, see \cite{Dubrovskii} and \cite{Wentzell-LD-Markov-Processes-1986}. Ellis and Rosen in 1980s derived precise asymptotics in the form (\ref{eq:PLD-Wiener-processes}) for Gaussian probability measures by suitable technical arguments based on Hilbert spaces, see \cite{Ellis-Rosen-1980}, \cite{Ellis-Rosen-I-1982}, \cite{Ellis-Rosen-II-1982}, \cite{Piterbarg-Fatalov-1995} and \cite{Fatalov-2012}. The purpose of this paper is to study precise asymptotics for large deviations in form (\ref{eq:PLD-Wiener-processes}) for a wide class of families of stochastic processes including diffusion processes, pure jump processes, deterministic processes and the mixture processes of them - locally infinitely divisible processes.

Most of the proofs regarding the precise large deviations mentioned above made use of a transformation
$\tilde{\mathbb{P}}^{\epsilon}(A)=\int_A\pi_{\phi}^{\epsilon}d\mathbb{P}$
($\pi_{\phi}^{\epsilon}$ is chosen such that
$\tilde{\mathbb{P}}^{\epsilon}(\Omega)=1$) in order that the main part of
$\mathbb{E}^{\epsilon}\left[\exp\{{\epsilon}^{-1}F(\xi^{\epsilon})\}\right]$
for small $\epsilon$ is due to the set of paths in a neighborhood of
$\phi$ which has large $\tilde{\mathbb{P}}^{\epsilon}$ probability. We call
such a transformation the \emph{generalized} Cram\'{e}r's
transformation which will be used also in this paper.

Back to the classical precise asymptotics for large deviations $\int_{\mathbb{R}}\exp\{nf(x)\}\mu^n(dx)$ on i.i.d. random variables, Cram\'{e}r's main tools are the asymptotic expansions on normal deviations for $\sqrt{n}\bar{S}_n$ in the form $F_{\sqrt{n}\bar{S}_n}(x)=F_{\infty}(x)+\sum_{i=1}^kP_i(x)n^{-i/2}+o(n^{-k/2})$ where $F_{\infty}(x)$ is the limiting distribution of $F_{\sqrt{n}\bar{S}_n}(x)$ (the distribution function of the random variable $\sqrt{n}\bar{S}_n$). Equivalently, for smooth function $g(x),$ the normal deviations take the following form
\begin{equation}\label{eq:normal-deviation-classical}
\mathbb{E}g(\sqrt{n}\bar{S}_n)=\mathbb{E}g(Y)+\sum_{i=1}^kp_in^{-i/2}+o(n^{-k/2})
\end{equation}
where $Y$ is the random variable corresponding to the distribution function $F_{\infty}(x).$ It is well-known that a family of stochastic processes $\eta^{\epsilon}$ converges weakly to a process $\eta$ as probability measures on the trajectory function space $\mathbb{X}$ if for any continuous and bounded functional $G(x)$ on $\mathbb{X}$
\begin{equation*}
\mathbb{E}^{\epsilon}G(\eta^{\epsilon})=\mathbb{E}G(\eta)+o(1).
\end{equation*}
The exact order for $o(1)$ is generally unknown. Thus, if one wants to follow the idea of Cram\'{e}r on random variables to derive precise asymptotics of large deviations for stochastic processes by using the asymptotic expansions on normal deviations for stochastic processes, then the first step would be to obtain normal deviations for stochastic processes, namely,
\begin{equation}\label{eq:normal-deviation-processes}
\mathbb{E}^{\epsilon}G(\eta^{\epsilon})=\mathbb{E}G(\eta)+\sum_{i=1}^kP_i\epsilon^{i/2}+o(\epsilon^{k/2}).
\end{equation}
But normal deviations (\ref{eq:normal-deviation-processes}) are far from clear until a recent result \cite{Yang-SPA-2012}, and we refer to \cite{Wentzell-Asymptotic-I-III} and the references therein for closely related works. The method of deriving precise asymptotics of large deviations from precise normal deviations for stochastic processes seems to appear for the first time in this paper.

In Section \ref{subsec:locally-infinitely-divisible} we give the definition of a locally infinitely divisible process and list several related concepts. The main result of this paper is contained in Section \ref{sec:main-theorem}, where some examples are also included. After appropriate recall from \cite{Yang-SPA-2012} on normal deviations for stochastic processes in Section \ref{subsec:normal-deviation}, we present the proof of our main theorem in the rest of Section \ref{sec:proof}.

As related problems, in Section \ref{sec:Math-Physics} we study the connections between precise asymptotics for large (or normal) deviations and for the solutions to partial integro-differential equations
\begin{equation*}
\begin{cases}
\begin{aligned}
\frac{\partial}{\partial t}u^{\epsilon}(t,x)=&\frac{\epsilon}{2}a(t,x)\Delta u^{\epsilon}(t,x)+b(t,x)\nabla u^{\epsilon}(t,x)+\epsilon^{-1}c(x)u^{\epsilon}(t,x)\\
&+\epsilon^{-1}\int_{\mathbb{R}}\left[u^{\epsilon}(t,x+\epsilon u)-u^{\epsilon}(t,x)-\epsilon u\nabla u^{\epsilon}(t,x)\right]\nu_{t,x}(du)
\end{aligned}\\
u^{\epsilon}(0,x)=g(x)
\end{cases}
\end{equation*}
under suitable smooth and growth conditions on $a,b,c$ and $g.$ For instance, if $c=1,$ $a(t,x)=a(x),$ $b(t,x)=b(x),$ $0<\inf_{x}a(x)\leq \sup_{x}a(x)<\infty,$ $\nu_{t,x}(du)=u^21_{\{|u|\leq1\}}(du),$ the smooth functions $a(x), b(x)$ and $g(x)$ are bounded together with their derivatives $d^j a/d x^j,$ $d^j b/d x^j$ and $d^j g/d x^j,$ then the precise asymptotics for the solution $u^{\epsilon}(t,x)$ for fixed $(t,x)$ is (with $n$ being an arbitrary integer)
$$u^{\epsilon}(t,x)=e^{t/{\epsilon}}\cdot\left[\sum_{k=0}^nk_i(x) \epsilon^{k/2}+o(\epsilon^{n/2})\right],\quad \text{ for constants }k_i \text{ depending on }x.$$

\subsection{Locally infinitely divisible processes}\label{subsec:locally-infinitely-divisible}
If $(\xi_t, \mathbb{P}_{s,x}),t\in[s,T],$ is a real-valued Markov process (the
subscript $_{s,x}$ means the process starts from $x$ at time $s$),
we use $P^{s,t},0\leq s\leq t\leq T,$ to denote the corresponding
multiplicative family of linear operators acting on functions
according to the formula
$$P^{s,t}f(x)=\mathbb{E}_{s,x}f(\xi_t),$$
where $\mathbb{E}_{s,x}$ is the expectation with respect to probability
measure $\mathbb{P}_{s,x}.$ The \emph{compensating operator} $\mathfrak{A}$
of this Markov process, taking functions $f(t,x)$ to functions of
the same two arguments, is defined by
\begin{align}\label{compensating-op-definition}
P^{s,t}f(t,\cdot)(x)=f(s,x)+\int_s^tP^{s,u}\mathfrak{A}f(u,\cdot)(x)du
\end{align}
under suitable assumptions on the measurability in $(t,x)$ of
$\mathfrak{A}f(t,x),$ where $P^{s,t}f(t,\cdot)(x)$ means that
$P^{s,t}$ is applied to the function $f(t,x)$ in its second argument
$x,$ and $P^{s,u}\mathfrak{A}f(u,\cdot)(x)$ means that $P^{s,u}$ is
applied to function $g(u,x):=\mathfrak{A}f(u,x)$ in its second
argument $x.$ If some measurability conditions are imposed on the
process $\xi_t(\omega),$ then (\ref{compensating-op-definition}) is
equivalent to that
$$f(t,\xi_t)-\int_s^t\mathfrak{A}f(u,\xi_u)du$$
is a martingale with respect to the natural family of
$\sigma$-algebras and every probability measure $P_{s,x}.$ Of
course, compensating operator $\mathfrak{A}$ is not defined
uniquely. Different versions are such that $\mathfrak{A}f(u,\xi_u)$
coincide almost surely except on a set of time argument $u$ of zero
Lebesgue measure.

We say $A_t$ is the generating operator of our process $(\xi_t,\mathbb{P}_{s,x})$ if for $s\leq t,$
$$P^{s,t}f(x)=f(x)+\int_s^tP^{s,u}A_uf(x)du$$
for suitable $f.$ Also a generating operator has different versions. For a wide class of Markov processes, a version of the compensating
operator $\mathfrak{A}$ of process $\xi_t$ for smooth functions
$f(t,x)$ is given by
$$\mathfrak{A}f(t,x)=\frac{\partial f}{\partial t}(t,x)+A_tf(t,\cdot)(x),$$
where generating operator $A_t$ acts on
functions of the spatial argument $x$ only.

For each fixed ${\epsilon}>0,$ let
$(\xi_t^{\epsilon}, \mathbb{P}_{0,x}^{\epsilon}),t\in[0,T],$ be a
one-dimensional process with jumps whose trajectories are right
continuous with left limits. We assume that the generating operator
of $\xi_t^{\epsilon}$ is
\begin{align}\label{infi-divi-section-1-2}
A_t^{\epsilon}f(x)={\epsilon}^{-1}\int_{\mathbb{R}}\left[f(x+{\epsilon}u)-f(x)-{\epsilon}uf'(x)\right]\nu_{t,x}(du)+\alpha(t,x)f'(x)+\frac{\epsilon}{2}a(t,x)f''(x)
\end{align}
for functions $f$ that are bounded and continuous together with their
first and second derivatives, and that a version of its compensating
operator is given by $\mathfrak{A}f(t,x)=\frac{\partial f}{\partial
t}(t,x)+A_tf(t,\cdot)(x)$ for bounded functions $f(t,x)$ that are
absolutely continuous in $t,$ twice continuously differentiable in
$x$ for fixed $t$ with bounded derivatives $\partial
f/\partial t, \partial f /\partial x$ and $\partial^2
f /\partial x^2.$ In order to make sense of the integral in
$A_t^{\epsilon}f(x)$ and also for the purpose of the proof, throughout this paper we impose two conditions on measures
$\nu_{t,x}:\,\int u^2\nu_{t,x}(du)<\infty$ for every $(t,x),$ and there is a
bounded support $K$ for all $\nu_{t,x}(\cdot),$ i.e.,
$\nu_{t,x}(K^c)\equiv0.$ The family $\{\xi_t^{\epsilon}\}$ is the underlying family of stochastic processes in this paper, and we call them \textit{locally infinitely divisible processes} with bounded support, see also \cite{Wentzell-Asymptotic-I-III} and \cite{Wentzell-LD-Markov-Processes-1986}. We note that this family contains diffusion processes, pure jump processes, deterministic processes and the processes coming from the mixture of them.

For each $\epsilon,$ the process $\xi^{\epsilon}$ makes jumps of
size $\epsilon\cdot u,$ according to the rate measure
${\epsilon}^{-1}\nu_{t,x}(du),$ and moves with velocity
$\alpha(t,x)-\int u\nu_{t,x}(du)$ between the jumps. We define
the cumulant $G^{\epsilon}(t,x;z)$ of
$(\xi_t^{\epsilon},\mathbb{P}_{0,x}^{\epsilon})$ by, for $t\in[0,T], x,z\in
\mathbb{R},$
$$G^{\epsilon}(t,x;z)=z\alpha(t,x)+\frac{\epsilon}{2}a(t,x)z^2+{\epsilon}^{-1}\int_{\mathbb{R}}\left(e^{z{\epsilon}u}-1-z{\epsilon}u\right)\nu_{t,x}(du).$$
Here $G^{\epsilon}(t,x;z)$ is well defined because of two conditions
we imposed on $\nu_{t,x},$ and it satisfies
$G^{\epsilon}(t,x;z)={\epsilon}^{-1}G_0(t,x;\epsilon z),$ where
$$G_0(t,x;z)=z\alpha(t,x)+\frac{1}{2}a(t,x)z^2+\int_{\mathbb{R}}\left(e^{zu}-1-zu\right)\nu_{t,x}(du).$$
Let $H_0(t,x;u),G_0(t,x;z)$ be coupled by the Legendre
transformation in the third argument,
$$H_0(t,x;u)=\sup_{z\in \mathbb{R}}\left[zu-G_0(t,x;z)\right].$$
For an absolutely continuous function $\phi_0$ (which will be
specified later as a maximizer) we define
$z^{\epsilon}(t)={\epsilon}^{-1}z_0(t), z_0(t)=\frac{\partial
H_0}{\partial u}(t,\phi_0(t);\phi'_0(t))$ and generalized
Cram\'{e}r's transformation:
$$\mathbb{P}_{0,x}^{z^{\epsilon}}(A):=\int_A\pi^{\epsilon}(0,T)d\mathbb{P}_{0,x}^{\epsilon},$$
with
$\pi^{\epsilon}(0,T)=\exp\left\{{\epsilon}^{-1}\int_0^Tz_0(t)d\xi_t^{\epsilon}-{\epsilon}^{-1}\int_0^TG_0(t,\xi_t^{\epsilon};z_0(t))dt\right\}.$
For each ${\epsilon}>0,$ this transformation gives us a new
probability measure $\mathbb{P}_{0,x}^{z^{\epsilon}}$ if we assume
$\pi^{\epsilon}(0,T)$ is a martingale as a process in $T$ with
respect to $\mathbb{P}_{0,x}^{\epsilon}.$ For
each ${\epsilon}>0,$ under $\mathbb{P}_{0,x}^{z^{\epsilon}}$ it turns out
$\xi^{\epsilon}$ is again a jump process with compensating operator
(see \cite{Wentzell-LD-Markov-Processes-1986}, Section 2.2.2),
\begin{equation}\label{sep-21-1}
\begin{aligned}
\mathfrak{A}^{z^{\epsilon}}f(t,x)=&\frac{\partial f}{\partial t}(t,x)+\frac{\partial G_0}{\partial z}(t,x;z_0(t))\frac{\partial f}{\partial x}(t,x)+\frac{\epsilon}{2}a(t,x)\frac{\partial^2 f}{\partial x^2}(t,x\\
&+{\epsilon}^{-1}\int_{\mathbb{R}}\left[f(t,x+{\epsilon}u)-f(t,x)-\epsilon u\frac{\partial f}{\partial x}(t,x)\right]e^{z_0(t)u}\nu_{t,x}(du).
\end{aligned}
\end{equation}

Let us define the normalized action functional as follows
$$S(\phi)=S_{0,T}(\phi)=\int_0^TH_0(t,\phi(t);\phi'(t))dt$$
for absolutely continuous function $\phi,$ otherwise $S(\phi)=+\infty.$ At the end of this section, we introduce several notations which are needed for our formulation of the main theorem. Let $\phi_0$ be continuously differentiable,
\begin{align*}
G_0^{*}(t,x;z)=&z\left[\alpha(t,\phi_0(t)+x)-\int_{\mathbb{R}}u\nu_{t,\phi_0(t)+x}(du)-\phi'_0(t)\right]\\
&+\frac{1}{2}a(t,\phi_0(t)+x)z^2+\int_{\mathbb{R}}\left(e^{zu}-1\right)e^{z_0(t)u}\nu_{t,\phi_0(t)+x}(du),
\end{align*}
and $H_0^{*}(t,x;u)$ be the Legendre transformation of
$G_0^{*}(t,x;z)$ in the third argument. For simplicity, throughout this paper we will only consider
$\mathbb{P}_{0,0}^{\epsilon},\mathbb{P}_{0,0}^{z^{\epsilon}},$ and use symbols
$\mathbb{P}^{\epsilon}, \mathbb{P}^{z^{\epsilon}}$ for short. The following symbols
are also used,
\begin{equation}\label{alphas-section-1-2}
\begin{aligned}
\alpha^1(t,x)=\frac{\partial G_0}{\partial z}(t,x;z_0(t)),\quad& \alpha^2(t,x)=a(t,x)+\int u^2e^{z_0(t)u}\nu_{t,x}(du);\\
\alpha^j(t,x)=\int
u^je^{z_0(t)u}\nu_{t,x}(du),\quad&\beta^j(t,x)=\int|u|^je^{z_0(t)u}\nu_{t,x}(du),j\geq3.
\end{aligned}
\end{equation}

\subsection{Functional derivatives}\label{subsec:functional-derivatives}
Let us include in this section the function spaces related to the trajectory spaces of our stochastic processes and the corresponding functional derivatives. We use $D_0[0,T]$ to denote the space of all functions defined on
$[0,T]$ vanishing at $0$ which are right continuous with left
limits; $C_0^1[0,T]$ the space of all continuously differentiable
functions on $[0,T]$ vanishing at $0;$ and $W_0^{1,2}[0,T]$ the
space of absolutely continuous functions vanishing at $0$ having
square integrable derivatives. In the space $D_0[0,T],$ the uniform norm
$||\phi||=\sup_{0\leq t\leq T}|\phi(t)|$ will be used. Throughout
this paper, we understand the differentiability of a
functional $F(\phi)$ on $D_0[0,T]$ as Fr\'{e}chet differentiability. Moreover, we assume that the derivatives
$F^{(j)}(\phi)(\delta_1,\cdots,\delta_j)$ can be represented as
integrals of the product
$\delta_1(s_1)\cdots\delta_j(s_j)$ with respect to some
signed measures, denoted by $F^{(j)}(\phi;\bullet)$:
\begin{align}\label{condition-derivative}
F^{(j)}(\phi)(\delta_1,\cdots,\delta_j)=\int_{[0,T]^j}\delta_1(s_1)\cdots\delta_j(s_j)F^{(j)}(\phi;ds_1\cdots
ds_j).
\end{align}
The norm of the signed measure is defined by
$$||F^{(j)}||:=\sup_{x[0,T]\in
D_0[0,T]}\left|F^{(j)}(x[0,T];\bullet)\right|([0,T]^j).$$
The notation $F^{(j)}(\phi)(y[0,T]^{\otimes^j})$ stands for the $j$-th
derivative $F^{(j)}(\phi)(y[0,T],\cdots,y[0,T])$ of the functional $F$
at point $\phi[0,T]$ in directions $y[0,T].$

\section{The main theorem and examples}\label{sec:main-theorem}
From now on, an integer $s\geq2$ will be used. To precisely state our main result, we make a list of assumptions on two pairs $G_0, H_0$ and $G_0^{*}, H_0^{*}$ introduced in Section \ref{subsec:locally-infinitely-divisible}. The first five general assumptions (A)-(E) can be found in \cite{Wentzell-LD-Markov-Processes-1986}. Let $p(t,x;z)$ and $q(t,x;u)$ be coupled by the Legendre transformation in the third arguments.

\textbf{(A)}. $p(t,x;z)\leq \overline{p}_0(z)$ for all
$t,x,z,$ where $\overline{p}_0$ is a downward convex non-negative
function, finite for all $z,$ and such that $\overline{p}_0(0)=0.$

\textbf{(B)}. $q(t,x;u)<\infty$ for the same $u$ for which
$\underline{q}_0(u)$ is finite, where $\underline{q}_0$ is the
Legendre transformation of $\overline{p}_0.$

\textbf{(C)}.
$$\mathop{\text{sup}}_{|t-s|<h,|x-y|<\delta,q(t,x;u)<\infty}\frac{q(s,y;u)-q(t,x;u)}{1+q(t,x;u)}\rightarrow0\quad\text{ for }{h}\downarrow0,\delta\downarrow0.$$

\textbf{(D)}. The set $\{u:\underline{q}_0(u)<\infty\}$ is open, and
$\sup_{t,x}q(t,x;u_0)<\infty$ for some point $u_0$ of it.

\textbf{(E)}. For any compactum
$U_K\subseteq\{u:\underline{q}_0(u)<\infty\},$ the partial derivative
$\frac{\partial q}{\partial u}(t,x;u)$ is bounded and
continuous in $u\in U_K$ uniformly with respect to all $t,x.$

Besides, one more technical assumption is imposed directly on $G_0$ and $\alpha^j.$

\textbf{(F)}. Let $G_0(t,x;z)$ and $H_0(t,x;u)$ be twice
differentiable with respect to $(x;z)$ and $(x;u)$ respectively, and
their first and second derivatives be continuous with respect to
$(t,x;z)$ and $(t,x;u).$ Furthermore, assume $G_0(t,x;z)$ is
differentiable in $x$ up to $s+1$ times,
$\frac{\partial^{s+1}G_0}{\partial x^{s+1}}(t,x;z)$ is bounded for
all $x\in \mathbb{R},t\in[0,T]$ and bounded $z.$ Suppose $\sup_{t,x,z}|\frac{\partial^2 G_0}{\partial z\partial
x}(t,x;z)|<\infty,$ $\inf_{t,x,z}|\frac{\partial^2 G_0}{\partial z^2}(t,x;z)|>0,$ $\sup_{t,x}|a(t,x)|<\infty$, and $||\alpha^1_{(i+1)}||+||\alpha^{j+1}_{(i)}||<\infty$ for
all integers $i,j\geq 1$ where the subscript $_{(i)}$ in
$\alpha^j_{(i)}(t,x):=\alpha^j_{\underbrace{22\cdots2}_{i}}(t,x)$
means the $i$-th partial derivative of $\alpha^j$ in its second
argument $x.$

\subsection{The main theorem}\label{subsec:main-theorem}
\begin{theorem}\label{main-thm} Let
$(\xi_t^{\epsilon},\mathbb{P}^{\epsilon}),t\in[0,T],$ be a family of
one-dimensional locally infinitely divisible processes with bounded support introduced in Section \ref{subsec:locally-infinitely-divisible}, For a continuous functional $F$ on $D_0[0,T]$ which is bounded above, let the maximum
of functional $F-S$ be attained at a unique function $\phi_0\in
C_0^1[0,T],$ and the assumption (F) be fulfilled. Furthermore, assume that the assumptions (A)-(E) are satisfied for $p(t,x;z)=G_0(t,x;z)$ and $p(t,x;z)=G_0^{*}(t,x;z).$

Suppose $F$ is $s+1$ times differentiable at all points $\phi$ in a
neighborhood of $\phi_0$ and $F^{(2)}(\phi_0)(x,x)<S^{(2)}(\phi_0)(x,x)$ for any non-zero
function $x\in W_0^{1,2}[0,T].$ For all $\phi$ in this neighborhood of $\phi_0,$ any $x[0,T],x_i[0,T]\in
D_0[0,T],$ we assume
\begin{equation}\label{condition-on-F}
\begin{aligned}
&F^{(2)}(\phi)(x[0,T],x[0,T])+\int_0^T(x(t))^2\frac{\partial^2
G_0}{\partial x^2}(t,\phi_0(t);z_0(t))dt\leq 0,\\
&\left|F^{(i)}(\phi)(x_1[0,T],\cdots,x_i[0,T])\right|\leq
p\left[\left|x_1(T)\cdots
x_i(T)\right|^m+\prod_{j=1}^{i}\left(1+\int_0^T\left|x_j(t)\right|^{n}dt\right)\right]
\end{aligned}
\end{equation}
with $2\leq i\leq s+1$ and some constants $m,n,p\geq 1.$ Another continuous and bounded functional $H$ is assumed to be $s-1$ times differentiable at all points
$\phi$ in a neighborhood of $\phi_0,$ and
\begin{align}\label{condition-on-H}
\left|H^{(i)}(\phi)(x_1[0,T],\cdots,x_i[0,T])\right|\leq
p'\left[\left|x_1(T)\cdots
x_i(T)\right|^{m'}+\prod_{j=1}^{i}\left(1+\int_0^T\left|x_j(t)\right|^{n'}dt\right)\right]
\end{align}
with $1\leq i\leq s-1$ and some positive constants $m',n',p'\geq 1,$ for all $x_i[0,T]\in
D_0[0,T].$

Then as $\epsilon\rightarrow0,$ the following precise asymptotics hold
\begin{align}\label{final-formula}
\mathbb{E}^{\epsilon}\left[H(\xi^{\epsilon})\exp\{{\epsilon}^{-1}F(\xi^{\epsilon})\}\right]=\exp\left\{{\epsilon}^{-1}\left[F(\phi_0)-S(\phi_0)\right]\right\}\left[\sum_{0\leq
i\leq (s-2)}K_i\cdot
{\epsilon}^{i/2}+o\left({\epsilon}^{(s-2)/2}\right)\right]
\end{align}
where the coefficients $K_i$ are determined by $F,H$ and their derivatives at $\phi_0;$ in particular,
\begin{align*}
&K_0=H(\phi_0)\cdot \mathbb{E}\left(\exp\{Q(2,\eta)\}\right),\\
&K_1=C_{01}H(\phi_0)+C_{11}+\mathbb{E}\left[\exp\{Q(2,\eta)\}\left(Q(3,\eta)H(\phi_0)+H^{(1)}(\phi_0)(\eta)\right)\right]
\end{align*}
with the constants $C_{ij}$ depending on $F,H$ and $\phi_0,$ the process $\eta$ being a Gaussian diffusion having diffusion coefficient $A(t)=\frac{\partial^2G_0}{\partial
z^2}(t,\phi_0(t);z_0(t))$ and drift coefficient
$B(t,x)=x\cdot\frac{\partial^2G_0}{\partial z
\partial x}\left(t,\phi_0(t);z_0(t)\right),$ and
$$Q(n,x[0,T])=\frac{1}{n!}F^{(n)}(\phi_0)(x[0,T],\cdots,x[0,T])+\int_0^T\frac{1}{n!}(x(t))^n\frac{\partial^n
G_0}{\partial x^n}(t,\phi_0(t);z_0(t))dt.$$
\end{theorem}

\textbf{Remark:} (1). The constants $C_{ij}$ are determined
from the asymptotic expansions for normal deviations formulated in (III) of Section \ref{subsec:Taylor-expansions-estimates}. The finiteness of the coefficients $K_i$ is proved in Section \ref{subsec:finiteness-coefficient}.

(2). In the spacial case when $\xi^{\epsilon}=\sqrt{\epsilon}W$, each coefficient $K_i$ in (\ref{final-formula}) with an odd index $i$ is equal to zero because of the symmetry of the distribution of $W.$

(3). If the initial position of $\xi^{\epsilon}$ is $x$ instead of $0,$ then we think of $\xi^{\epsilon}-x$ as a new process with zero initial position and $$\mathbb{E}^{\epsilon}_{0,x}\Phi(\xi^{\epsilon})=\mathbb{E}^{\epsilon}_{0,0}\Phi(\xi^{\epsilon}-x).$$

\subsection{Examples}\label{subsec:examples}
\textbf{Example 1}. Let us consider a family of
one-dimensional pure jump processes $\xi_t^{\epsilon},t\in[0,1],$
with generating operator given by
$$A_t^{\epsilon}f(x)={\epsilon}^{-1}\int_{\mathbb{R}}\left(f(x+{\epsilon}u)-f(x)\right)\nu_{t,x}(du),$$
where
$\nu_{t,x}(du)=\frac{1}{2}\left(\delta_1(du)+\delta_{-1}(du)\right)$
with $\delta_1(\cdot)$ (resp. $\delta_{-1}(\cdot)$) denoting the
probability measure concentrating at point $1$ (resp. $-1$). For
each $\epsilon,$ the trajectories of $\xi^{\epsilon}$ are step
functions with finitely many steps on $[0,1].$ This process makes
jumps of size $\pm\epsilon$ according to the rate
$\frac{1}{2}{\epsilon}^{-1}.$

The most probable trajectory for $\xi^{\epsilon}$ as
$\epsilon\rightarrow0$ is identically zero. To see this, we first
note that
\begin{align*}
G_0(t,x;z)&=\int_{\mathbb{R}}\left(e^{zu}-1\right)\nu_{t,x}(du)=\frac{1}{2}\left(e^z+e^{-z}-2\right);\\
H_0(t,x;u)&=\sup_{z\in
\mathbb{R}}\left[zu-G_0(t,x;z)\right]=u\ln\left(u+\sqrt{u^2+1}\right)+1-\sqrt{u^2+1};\\
S(\phi)&=\int_0^1H_0(t,\phi(t);\phi'(t))dt,\text{ for absolutely
continuous }\phi\text{ in }D_0[0,1].
\end{align*}
As a function of $u,$ $H_0(t,x;u)$ is positive except at $u=0,$
strictly increasing on $(0,\infty),$ and strictly decreasing on
$(-\infty,0).$ Thus, in order to make $S(\phi)=0,$ it is required
$\phi'(t)=0$ almost everywhere with respect to Lebesgue measure. But
$\phi(t)$ is absolutely continuous and $\phi(0)=0,$ it follows that
$\phi(t)\equiv0.$ This proves that the most probable trajectory is
zero.

Let the functional $F$ on $D_0[0,1]$ be
$$F(\phi)=\int_0^1\left(\phi(t)-\phi^2(t)\right)dt,$$
and $H(\phi)\equiv1.$ We need to show that $\max(F-S)$ is attained
at a unique function $\phi_0\in C_0^1[0,1],$ that is, the following
variational problem
\begin{align}\label{variational-problem}
\max_{\phi\in
C_0^1[0,1]}\int_0^1\left[\phi(t)-\phi(t)^2-\left(\phi'(t)\ln\left(\phi'(t)+\sqrt{\phi'(t)^2+1}\right)+1-\sqrt{\phi'(t)^2+1}\right)\right]dt,
\end{align}
has an unique (nonzero) solution. The existence and uniqueness for a
nonzero solution of (\ref{variational-problem}) are shown in the Appendix. All other conditions of Theorem \ref{main-thm} can be
easily checked.

We could have considered some wider families of processes and more
general functionals (e.g.
$F(\phi[0,T])=h\left(\int_0^Tg\left(\phi\left(s\right)\right)ds\right)$),
and each time we will have to verify the existence and uniqueness of
the solution for the corresponding variational problem.\\
\textbf{Example 2}. Suppose $\xi_t^{\epsilon},t\in[0,1]$ is a family
of one-dimensional pure jump processes with generating operator
$$A_t^{\epsilon}f(x)={\epsilon}^{-1}\int_{\mathbb{R}}\left(f(x+{\epsilon}u)-f(x)\right)\nu_{t,x}(du),$$
where $\nu_{t,x}(du)=r(x)\delta_1(du)+l(x)\delta_{-1}(du)$ and
$r(x)=l(x)=\sin(x)+2.$ It is easy to get
\begin{align*}
G_0(t,x;z)&=\int_{\mathbb{R}}\left(e^{zu}-1\right)\nu_{t,x}(du)=r(x)(e^z-1)+l(x)(e^{-z}-1).
\end{align*}
Again we consider functional
$F(\phi)=\int_0^1\left(\phi(t)-\phi^2(t)\right)dt.$ The existence and uniqueness of the variational problem $\max_{\phi\in C_0^1[0,1]}[F(\phi)-S(\phi)]$ can be similarly obtained as Example 1. Now we check $G_0^{*}(t,x;z)$
satisfies conditions (A)-(E). Let us recall that in this example
\begin{align*}
G_0^{*}(t,x;z)=&-z[\phi_0'(t)+r(\phi_0(t)+x)-l(\phi_0(t)+x)]\\
&+r(\phi_0(t)+x)(e^z-1)e^{z_0(t)}+l(\phi_0(t)+x)(e^{-z}-1)e^{-z_0(t)}.
\end{align*}
For conditions (A) and (B): we choose $\bar{G}_0(z)=C(e^{|z|}-1)$
with a positive constant $C,$ then $\underline{H}_0(u)=Ch(|u|/C),$
where function $h(y)=y\ln y-y+1$ for $y\geq1,$ and $=0$ for $0\leq
y<1.$ For condition (C): for any fixed $h>0,$ if $(t,x)$ and
$(s,y)$ are close enough, then
\begin{align}\label{cccccccc}
G_0^{*}(t,x;(1-h)z)-(1-h)G_0^{*}(s,y;z)\leq h.
\end{align}
To see (\ref{cccccccc}), note that for large $z\rightarrow\infty$ or
$z\rightarrow-\infty,$ the left hand side of (\ref{cccccccc}) goes
to $-\infty$ uniformly in $t$ and $x.$ This means that we just need
to consider bounded $z,$ which proves (\ref{cccccccc}). Conditions (D) and
(E) are easy to be checked.

\section{Proof of Theorem \ref{main-thm}}\label{sec:proof}
As mentioned in the introduction, we will use precise normal deviations for stochastic processes in our proof. The normal deviations needed here are not for the processes $\xi^{\epsilon},$ but for another family of processes $\eta^{\epsilon}$ related to $\xi^{\epsilon}.$

Because of assumption (F) in Section \ref{sec:main-theorem}, the ordinary differential equation $x'(t)=\alpha^1(t,x(t))$ with an initial condition $x(0)=0$ has a unique solution with $\alpha^1$ defined in (\ref{alphas-section-1-2}). Furthermore, this unique solution can be proved to be $\phi_0$ from Legendre transformation. Let us set
$$\eta^{\epsilon}_t=\epsilon^{-1/2}(\xi^{\epsilon}_t-\phi_0(t)).$$
Note that here the initial point $\eta^{\epsilon}_0=0.$ More generally, we consider an initial point $\eta^{\epsilon}_0=x$ in this section in order to fully exhibit the normal deviations. It was proved in \cite{Wentzell-Asymptotic-I-III} that as $\epsilon\rightarrow0$ the family $\eta^{\epsilon}$ under $\mathbb{P}^{z^{\epsilon}}_{0,x}$ converges weakly to a process
$\eta$, which is a Gaussian diffusion
process on the real line with generating operator
\begin{align}\label{generator-of-eta-00}
A^{\eta}_tf(x)=\alpha_2^1(t,\phi_0(t))\cdot x\cdot
f'(x)+\frac{1}{2}\alpha^2(t,\phi_0(t))\cdot f''(x),
\end{align} where
the subscript $_2$ means differentiation in the second spatial argument.

\subsection{Results on normal deviations}\label{subsec:normal-deviation}
Before the statement of normal deviations, let us recall a differential operator $A_1$
which was defined in \cite{Wentzell-Asymptotic-I-III} for functionals
$G$ on $D[0,T]$ (consisting of right continuous functions with left limits):
\begin{align*}
A_1G(x[0,T])&=\sum_{k=1}^3\int_{[0,T]^k}\Gamma_1^k(x[0,T];s_1,\cdots,s_k)G^{(k)}(x[0,T];ds_1\cdots
ds_k)
\end{align*}
where
\begin{align*}
\Gamma_1^1(x[0,T];s_1)&=\frac{1}{2}\int_0^{s_1}\alpha_{22}^1(t,\phi_0(t))x(t)^2\exp\left\{\int_0^{s_1}\alpha_2^1(v,\phi_0(v))dv\right\}dt,\\
\Gamma_1^2(x[0,T];s_1,s_2)&=\frac{1}{2}\int_0^{\min\{s_1,s_2\}}\alpha_{2}^2(t,\phi_0(t))x(t)\exp\left\{\sum_{i=1}^2\int_0^{s_i}\alpha_2^1(v,\phi_0(v))dv\right\}dt,\\
\Gamma_1^3(x[0,T];s_1,s_2,s_3)&=\frac{1}{6}\int_0^{\min\{s_1,s_2,s_3\}}\alpha^3(t,\phi_0(t))\exp\left\{\sum_{i=1}^3\int_0^{s_i}\alpha_2^1(v,\phi_0(v))dv\right\}dt.
\end{align*}
A crucial functional in \cite{Wentzell-Asymptotic-I-III} and \cite{Yang-SPA-2012} is defined through a conditional expectation on the past path
\begin{align*}
f(t,x[0,t])=\mathbb{E}_{t,x[0,t]}G(\eta),\quad t\in[0,T].
\end{align*}

\begin{theorem}[Theorem 5.2 in \cite{Yang-SPA-2012}]\label{theorem-from-SPA}
Let a functional $G(x[0,T])$ on $D[0,T]$ be $3(s-2)$ times
differentiable with the following conditions:

(I). there is a constant $C>0$ such that for all $x[0,T],
y[0,T]\in D[0,T]$
\begin{align*}
\left|G(x[0,T])\right|&\leq
C\left(1+|x(T)|^{s}+\int_0^T|x(t)|^{s}dt\right),\\
\left|G^{(i)}(x[0,T])(y[0,T]^{\otimes^i})\right|&\leq (1+||y||^i)
C\left(1+|x(T)|^{s-2}+\int_0^T|x(t)|^{s-2}dt\right), 1\leq i\leq3(s-2),
\end{align*}

(II). $G^{(i)}(x[0,T])(I_{[t,T]}\delta,\cdots,I_{[t,T]}\delta),1\leq
i\leq3(s-2),$ are continuous with respect to $x[0,T]$ uniformly as
$x[0,T]$ changes over an arbitrary compact subset of $D[0,T],$
$t$ over $[0,T],$ and $\delta[0,T]$ over the set of Lipschitz
continuous functions with constant $1,$ $||\delta||\leq 1.$

Then as $\epsilon\rightarrow0,$ under the assumptions of Theorem \ref{main-thm} the precise normal deviations hold
\begin{align}\label{main-proving-formula}
\mathbb{E}^{z^{\epsilon}}_{0,x}G(\eta^{\epsilon})=\mathbb{E}_{0,x}G(\eta)+\sum_{i=1}^{s-2}\epsilon^{i/2}\mathbb{E}_{0,x}A_iG(\eta)+o(\epsilon^{(s-2)/2})
\end{align}
where $A_1$ is a third-order differential operator defined before,
$A_2$ is a sixth-order differential operator given by
\begin{align*}
A_2&G(x[0,T])=\int_0^TA_1\widetilde{G}(x[0,t])dt+\int_0^T\Big[\frac{1}{3!}\alpha^1_{222}(t,\phi_0(t))x(t)^3f^{(1)}(t,x[0,t])(I_{\{t\}})\\
&+\frac{1}{4}\alpha^2_{22}(t,\phi_0(t))x(t)^2f^{(2)}(t,x[0,t])(I_{\{t\}}^{\otimes^2})+\frac{1}{3!}\alpha^3_{2}(t,\phi_0(t))x(t)f^{(3)}(t,x[0,t])(I_{\{t\}}^{\otimes^3})\\
&+\frac{1}{4!}\alpha^4(t,\phi_0(t))f^{(4)}(t,x[0,t])(I_{\{t\}}^{\otimes^4})\Big]dt
\end{align*}
with
\begin{align*}
\widetilde{G}(x[0,t])=&\frac{1}{2}\alpha^1_{22}(t,\phi_0(t))x^2(t)f^{(1)}(t,x[0,t])(I_{\{t\}})+\frac{1}{2}\alpha^2_{2}(t,\phi_0(t))x(t)f^{(2)}(t,x[0,t])(I_{\{t\}}^{\otimes^2})\notag\\
&+\frac{1}{6}\alpha^3(t,\phi_0(t))f^{(3)}(t,x[0,t])(I_{\{t\}}^{\otimes^3}),
\end{align*}
and $A_3,\cdots,A_{s-2}$ are suitable differential operators defined
through derivatives of $f(t,x[0,t]).$
\end{theorem}

\subsection{Large deviations for $\epsilon^{1/2}\eta^{\epsilon}$}\label{subsec:large-deviation-for-eta}
It can be easily seen that the process
$\eta^{\epsilon}=\epsilon^{-1/2}(\xi^{\epsilon}-\phi_0)$ under the measure
$\mathbb{P}^{z^{\epsilon}},$ has compensating operator
\begin{equation}\label{compen-eta-h}
\begin{aligned}
\mathfrak{A}^{\eta^{\epsilon}}f(t,x)&=\frac{\partial f}{\partial
t}(t,x)+\epsilon^{-1/2}\left[\frac{\partial G_0}{\partial
z}(t,\phi_0(t)+\epsilon^{1/2}x;z_0(t))-\phi_0'(t)\right]\frac{\partial
f}{\partial x}(t,x)\\
&\qquad+\frac{1}{2}a(t,\phi_0(t)+\epsilon^{1/2}x)\frac{\partial^2
f}{\partial x^2}(t,x)\\
&\qquad+
{\epsilon}^{-1}\int_{\mathbb{R}}\left[f(t,x+\epsilon^{1/2}u)-f(t,x)-\epsilon^{1/2}\frac{\partial
f}{\partial x}(t,x)\cdot
u\right]e^{z_0(t)u}\nu_{t,\phi_0(t)+\epsilon^{1/2}x}(du),
\end{aligned}
\end{equation}
and cumulant
\begin{align*}
G^{\eta^{\epsilon}}(t,x;z)=&z\cdot
\epsilon^{-1/2}\left[\frac{\partial G_0}{\partial
z}(t,\phi_0(t)+\epsilon^{1/2}x;z_0(t))
-\phi_0'(t)\right]+\frac{1}{2}a(t,\phi_0(t)+\epsilon^{1/2}x)z^2\\
&+{\epsilon}^{-1}\int_{\mathbb{R}}\left[e^{z\sqrt{\epsilon}u}-1-z\sqrt{\epsilon}u\right]e^{z_0(t)u}\nu_{t,\phi_0(t)+\epsilon^{1/2}x}(du)\notag.
\end{align*}
The limiting process $\eta$ has compensating operator given by
\begin{align*}
\mathfrak{A}^{\eta}f(t,x)=\frac{\partial f}{\partial
t}(t,x)+x\cdot\frac{\partial^2 G_0}{\partial z
\partial x}(t,\phi_0(t);&z_0(t))\cdot\frac{\partial
f}{\partial x}(t,x)+ \frac{1}{2}\frac{\partial^2 G_0}{\partial
z^2}(t,\phi_0(t);z_0(t))\cdot\frac{\partial^2 f}{\partial x^2}(t,x),
\end{align*}
and cumulant
\begin{align*}
G^{\eta}(t,x;z)&=zx\cdot\frac{\partial^2 G_0}{\partial z
\partial x}(t,\phi_0(t);z_0(t))+\frac{1}{2}z^2\cdot\frac{\partial^2 G_0}{\partial
z^2}(t,\phi_0(t);z_0(t)).
\end{align*}
In this case, the Legendre transformation of $G^{\eta}(t,x;z)$ in
$z$ becomes
\begin{align*}
H^{\eta}(t,x;u)&=\frac{1}{2}\frac{\partial^2 H_0}{\partial
u^2}(t,\phi_0(t);\phi_0'(t))\left(u-x\cdot\frac{\partial^2
G_0}{\partial z
\partial x}(t,\phi_0(t);z_0(t))\right)^2.
\end{align*}
It then follows from Section 5.2.6 of \cite{Wentzell-LD-Markov-Processes-1986} that the normalized action functional for the family of
processes $\epsilon^{1/2}\eta$ is
\begin{align*}
I^{\eta}(f(\cdot))=\int_0^TH^{\eta}(t,f(t);f'(t))dt.
\end{align*}
An important tool we need for the proof of Theorem~\ref{main-thm} is
the following convergence
\begin{align}\label{eponential-conv-zero}
\mathbb{P}^{z^{\epsilon}}\left\{||\eta^{\epsilon}||\geq
\epsilon^{-1/2}\right\}\rightarrow0\text{ exponentially
fast as }\epsilon\downarrow0.
\end{align}

\subsubsection{Proof
of~(\ref{eponential-conv-zero})}\label{proof-finite-under-H-2}
Let us consider the family of processes
$\zeta^{\epsilon}=\epsilon^{1/2}\eta^{\epsilon}=\xi^{\epsilon}-\phi_0$
with respect to probabilities $\mathbb{P}^{z^{\epsilon}}.$ Since $\eta^{\epsilon}$ converge weakly to $\eta,$ it is expected that the the most
probable trajectory of $\zeta^{\epsilon}$ as ${\epsilon}\downarrow0$
is path zero. Then the exponential convergence to zero of
$\mathbb{P}^{z^{\epsilon}}\left\{||\eta^{\epsilon}||\geq
\epsilon^{-1/2}\right\}=\mathbb{P}^{z^{\epsilon}}\left\{||\zeta^{\epsilon}||\geq
1\right\}$ follows provided a large deviation result for
$\zeta^{\epsilon}$ is proved. To be precise, we now prove that a large principle holds for
$\zeta^{\epsilon}$ under the assumptions of Theorem \ref{main-thm}

First it is straightforward to compute the cumulant of $\zeta^{\epsilon}:$
\begin{align*}
G^{\zeta^{\epsilon}}(t,x;z)=&z\left[\frac{\partial G_0}{\partial
z}(t,\phi_0(t)+x;z_0(t))-\phi'_0(t)\right]+\frac{\epsilon}{2}a(t,\phi_0(t)+x)z^2\\
&+{\epsilon}^{-1}\int_{\mathbb{R}}\left[e^{z{\epsilon}u}-1-z{\epsilon}u\right]e^{z_0(t)u}\nu_{t,\phi_0(t)+x}(du),
\end{align*}
which satisfies $\epsilon
G^{\zeta^{\epsilon}}(t,x;{\epsilon}^{-1}z)=G_0^{*}(t,x;z).$ Then by taking into account the assumptions (A)-(F), we can deduce a large deviation principle with the following normalized action functional
$$S_{0,T}^{*}(\phi)=\int_0^TH_0^{*}(t,\phi(t);\phi'(t))dt,$$
for absolutely continuous function $\phi$ (see Theorem 3.2.1 in \cite{Wentzell-LD-Markov-Processes-1986} for details).

Now let us consider a closed set $A$ in $D_0[0,T]$ given by
$A=\{x[0,T]:||x[0,T]||\geq 1\}.$ The large deviation principle
for $\zeta^{\epsilon}$ gives that, for any small $\gamma>0,$ there is $\epsilon_0$ such that for all $\epsilon\in(0,\epsilon_0),$
\begin{align*}
\mathbb{P}^{z^{\epsilon}}\left\{||\eta^{\epsilon}||\geq
\epsilon^{-1/2}\right\}&=\mathbb{P}^{z^{\epsilon}}\left\{||\zeta^{\epsilon}||\geq
1\right\}=\mathbb{P}^{z^{\epsilon}}\left\{\zeta^{\epsilon}\in A\right\}\leq\exp\left\{-{\epsilon}^{-1}[\inf_{\phi\in
A}S_{0,T}^{*}(\phi)-\gamma]\right\}.
\end{align*}
So (\ref{eponential-conv-zero}) is proved if $C_A:=\inf_{\phi\in
A}S_{0,T}^{*}(\phi)>0.$

Note that $C_A$ is reached at some point $\phi_A.$ This is because
$A$ is closed and any level set $\left\{\phi\in
D_0[0,T]:S_{0,T}^{*}(\phi)\leq s\right\}$ is compact for any $s>0.$
Now we set $C_A=S_{0,T}^{*}(\phi_A).$ If $S_{0,T}^{*}(\phi_A)=0,$
then $H_0^{*}(t,\phi_A(t);\phi_A'(t))=0$ almost everywhere with
respect to Lebesgue measure, i.e., for all $z\in \mathbb{R},$
\begin{align*}
z\cdot \phi_A'(t)-G_0^{*}(t,\phi_A(t);z)\leq 0,\text{ almost all }t.
\end{align*}
Thus for almost all $t,$
\begin{align*}
&\phi_A'(t)\leq\lim_{z\downarrow0}z^{-1}G_0^{*}(t,\phi_A(t);z)=\left[\frac{\partial
G_0}{\partial
z}(t,\phi_0(t)+\phi_A(t);z_0(t))-\phi'_0(t)\right]\\
&+\lim_{z\downarrow0}z^{-1}\int_{\mathbb{R}}\left(e^{zu}-1-zu\right)e^{z_0(t)u}\nu_{t,\phi_0(t)+\phi_A(t)}(du)=\frac{\partial
G_0}{\partial z}(t,\phi_0(t)+\phi_A(t);z_0(t))-\phi'_0(t),
\end{align*}
where the fact that the second limit is equal to zero is from the
assumption that $\nu_{t,x}$ have a bounded support $K.$ Similarly,
\begin{align*}
\phi_A'(t)\geq\lim_{z\uparrow0}z^{-1}G_0^{*}(t,\phi_A(t);z)=\frac{\partial
G_0}{\partial z}(t,\phi_0(t)+\phi_A(t);z_0(t))-\phi'_0(t),
\end{align*}
thus $\phi_A'(t)=\frac{\partial G_0}{\partial
z}(t,\phi_0(t)+\phi_A(t);z_0(t))-\phi'_0(t).$ Taking the initial
condition $\phi_A(0)=0$ into account, we deduce that
$\phi_A\equiv0.$ But this is a contradiction with $||\phi_A||\geq1,$
so $S_{0,T}^{*}(\phi_A)\neq0,$ i.e. $C_A>0.$

\subsection{Taylor's expansions and estimates}\label{subsec:Taylor-expansions-estimates}
We start this section with a technical lemma which suggests that when we consider the precise asymptotics for large deviations, the part away from the most probable path can be simply dropped.
\begin{lemma}\label{lemma-01}
Let the family $\xi^{\epsilon}$ satisfy a large deviation principle with a normalized action functional $S,$ and $F$ be a continuous measurable functional on $D_0[0,T].$ Suppose $F$ is bounded above and the difference $F-S$ attains its maximum at a unique function
$\phi_0\in D_0[0,T].$ Then for any $h>0,$ there is a $\gamma>0$ such
that as ${\epsilon}\rightarrow0,$
$$\mathbb{E}^{\epsilon}\left[1_{\{||\xi^{\epsilon}-\phi_0||\geq h\}}\exp\{{\epsilon}^{-1}F(\xi^{\epsilon})\}\right]=o\left(\exp\left\{{\epsilon}^{-1}[F(\phi_0)-S(\phi_0)-\gamma]\right\}\right).$$
\end{lemma}

The proof of this lemma is contained in the Appendix. It follows from this lemma that for any
$h>0,$ there exists some $\gamma>0$ such that as
${\epsilon}\rightarrow0,$
\begin{align*}
\mathbb{E}^{\epsilon}\left[H(\xi^{\epsilon})\exp\{{\epsilon}^{-1}F(\xi^{\epsilon})\}\right]=&\mathbb{E}^{\epsilon}\left[1_{\left\{||\xi^{\epsilon}-\phi_0||<h\right\}}H(\xi^{\epsilon})\exp\{{\epsilon}^{-1}F(\xi^{\epsilon})\}\right]\\
&+o\left(\exp\left\{{\epsilon}^{-1}\left[F(\phi_0)-S(\phi_0)-\gamma\right]\right\}\right).
\end{align*}
Noticing that $\exp\{-{\epsilon}^{-1}\gamma\}$ tends to zero exponentially fast, we thus only focus on the first part over the set $\left\{||\xi^{\epsilon}-\phi_0||<h\right\}.$ Simple calculations yield
\begin{equation}\label{first-step-1.5}
\begin{aligned}
&\mathbb{E}^{\epsilon}\left[1_{\left\{||\xi^{\epsilon}-\phi_0||<h\right\}}H(\xi^{\epsilon})\exp\{{\epsilon}^{-1}F(\xi^{\epsilon})\}\right]\\
&=\mathbb{E}^{z^{\epsilon}}\left[1_{\left\{||\xi^{\epsilon}-\phi_0||<h\right\}}H(\xi^{\epsilon})\exp\left\{{\epsilon}^{-1}F(\xi^{\epsilon})-{\epsilon}^{-1}\int_0^Tz_0(t)d\xi_t^{\epsilon}+{\epsilon}^{-1}\int_0^TG_0(t,\xi_t^{\epsilon};z_0(t))dt\right\}\right]\\
&=\mathbb{E}^{z^{\epsilon}}\Big[1_{\left\{||\eta^{\epsilon}||<\epsilon^{-1/2}h\right\}}H(\phi_0+{\epsilon}^{1/2}\eta^{\epsilon})\exp\Big\{{\epsilon}^{-1}F(\phi_0+{\epsilon}^{1/2}\eta^{\epsilon})-{\epsilon}^{-1/2}\int_0^Tz_0(t)d\eta_t^{\epsilon}\\
&\qquad\qquad-{\epsilon}^{-1}\int_0^T\Big(z_0(t)\phi'_0(t)-G_0(t,\phi_0(t)+{\epsilon}^{1/2}\eta_t^{\epsilon};z_0(t))\Big)dt\Big\}\Big].
\end{aligned}
\end{equation}
Now we apply Taylor's expansion for $F$ at $\phi_0$ up to order $s$ with an integral's remainder ($IR$) (see \cite{Cartan-1971} for
details),
\begin{align*}
&F(\phi_0+{\epsilon}^{1/2}\eta^{\epsilon})=F(\phi_0)+{\epsilon}^{1/2}F^{(1)}(\phi_0)(\eta^{\epsilon})+\cdots+\frac{{\epsilon}^{s/2}}{s!}F^{(s)}(\phi_0)(\eta^{\epsilon},\cdots,\eta^{\epsilon})+IR_1,\\
&\qquad\qquad\qquad\qquad\qquad\qquad IR_1={\epsilon}^{\frac{s+1}{2}}\int_0^1\frac{(1-u)^s}{s!}F^{(s+1)}(\phi_0+u{\epsilon}^{1/2}\eta^{\epsilon})(\eta^{\epsilon},\cdots,\eta^{\epsilon})du.
\end{align*}
For $G_0,$ Taylor's expansion at $\phi_0(t)$ in the second argument up to order $s$ with an integral's remainders yields
\begin{align*}
&\int_0^TG_0(t,\phi_0(t)+{\epsilon}^{1/2}\eta_t^{\epsilon};z_0(t))dt=\\
&\int_0^T\left(G_0(t,\phi_0(t);z_0(t))+{\epsilon}^{1/2}\eta_t^{\epsilon}\frac{\partial
G_0}{\partial
x}(t,\phi_0(t);z_0(t))+\cdots+\frac{\left({\epsilon}^{1/2}\eta_t^{\epsilon}\right)^s}{s!}\frac{\partial^s
G_0}{\partial
x^s}(t,\phi_0(t);z_0(t))\right)dt\\
&\qquad\qquad+IR_2
\end{align*}
where
$$
IR_2=\int_0^T\left({\epsilon}^{1/2}\eta_t^{\epsilon}\right)^{s+1}\int_0^1\frac{(1-u)^s}{s!}\frac{\partial^{s+1}
G_0}{\partial
x^{s+1}}(t,\phi_0(t)+u{\epsilon}^{1/2}\eta_t^{\epsilon};z_0(t))dudt.$$
We use these expansions to replace the exponential in
(\ref{first-step-1.5}) to get
\begin{equation}\label{first-step-2}
\begin{aligned}
&\exp\left\{{\epsilon}^{-1}F(\phi_0+\epsilon^{1/2}\eta^{\epsilon})-\epsilon^{-1/2}\int_0^Tz_0(t)d\eta_t^{\epsilon}\right.\\
&\qquad\qquad\qquad\qquad\qquad\quad\left.-{\epsilon}^{-1}\int_0^T\Big(z_0(t)\phi'_0(t)-G_0(t,\phi_0(t)+\epsilon^{1/2}\eta_t^{\epsilon};z_0(t))\Big)dt\right\}\\
&=\exp\Big\{{\epsilon}^{-1}\left[F(\phi_0)-\int_0^T\Big(z_0(t)\phi'_0(t)-G_0(t,\phi_0(t);z_0(t))\Big)dt\right]\\
&\qquad\qquad+{\epsilon}^{-1/2}\left[F^{(1)}(\phi_0)(\eta^{\epsilon})-\int_0^Tz_0(t)d\eta_t^{\epsilon}+\int_0^T\eta_t^{\epsilon}\frac{\partial G_0}{\partial x}(t,\phi_0(t);z_0(t))dt\right]\\
&\qquad\qquad+\left[\frac{1}{2!}F^{(2)}(\phi_0)(\eta^{\epsilon},\eta^{\epsilon})+\int_0^T\frac{1}{2!}(\eta_t^{\epsilon})^2\frac{\partial^2 G_0}{\partial x^2}(t,\phi_0(t);z_0(t))dt\right]\\
&\qquad\qquad+{\epsilon}^{1/2}\left[\frac{1}{3!}F^{(3)}(\phi_0)(\eta^{\epsilon},\eta^{\epsilon},\eta^{\epsilon})+\int_0^T\frac{1}{3!}(\eta_t^{\epsilon})^3\frac{\partial^3 G_0}{\partial x^3}(t,\phi_0(t);z_0(t))dt\right]+\cdots\\
&\qquad\qquad+{\epsilon}^{\frac{s-2}{2}}\left[\frac{1}{s!}F^{(s)}(\phi_0)(\eta^{\epsilon},\cdots,\eta^{\epsilon})+\int_0^T\frac{1}{s!}(\eta_t^{\epsilon})^s\frac{\partial^s
G_0}{\partial
x^s}(t,\phi_0(t);z_0(t))dt\right]\\
&\qquad\qquad+{\epsilon}^{-1}(IR_1+IR_2)\Big\}\\
&=\exp\Big\{{\epsilon}^{-1}\left[F(\phi_0)-S(\phi_0)\right]+\left[\frac{1}{2!}F^{(2)}(\phi_0)(\eta^{\epsilon},\eta^{\epsilon})+\int_0^T\frac{1}{2!}(\eta_t^{\epsilon})^2\frac{\partial^2 G_0}{\partial x^2}(t,\phi_0(t);z_0(t))dt\right]\\
&\qquad\qquad+{\epsilon}^{1/2}\left[\frac{1}{3!}F^{(3)}(\phi_0)(\eta^{\epsilon},\eta^{\epsilon},\eta^{\epsilon})+\int_0^T\frac{1}{3!}(\eta_t^{\epsilon})^3\frac{\partial^3 G_0}{\partial x^3}(t,\phi_0(t);z_0(t))dt\right]+\cdots\\
&\qquad\qquad+{\epsilon}^{\frac{s-2}{2}}\left[\frac{1}{s!}F^{(s)}(\phi_0)(\eta^{\epsilon},\cdots,\eta^{\epsilon})+\int_0^T\frac{1}{s!}(\eta_t^{\epsilon})^s\frac{\partial^s
G_0}{\partial
x^s}(t,\phi_0(t);z_0(t))dt\right]\\
&\qquad\qquad+{\epsilon}^{-1}(IR_1+IR_2)\Big\}\\
&=\exp\left\{{\epsilon}^{-1}\left[F(\phi_0)-S(\phi_0)\right]+Q(2,\eta^{\epsilon})\right\}\times\\
&\qquad\exp\Big\{{\epsilon}^{1/2}Q(3,\eta^{\epsilon})+\cdots+{\epsilon}^{\frac{s-2}{2}}Q(s,\eta^{\epsilon})+{\epsilon}^{-1}(IR_1+IR_2)\Big\}.
\end{aligned}
\end{equation}
For the second exponential function in the last equality of
(\ref{first-step-2}), we apply
$e^a=1+a+\cdots+a^{s-2}/(s-2)!+\frac{e^{\theta(a)\cdot
a}}{(s-1)!}a^{s-1}$ with $0\leq\theta(a)\leq1,$ then formula
(\ref{first-step-2}) is equal to
\begin{align*}
(\ref{first-step-2})=&\exp\left\{{\epsilon}^{-1}\left[F(\phi_0)-S(\phi_0)\right]+Q(2,\eta^{\epsilon})\right\}\times\\
&\left\{1+\epsilon^{1/2}Q(3,\eta^{\epsilon})+\epsilon\left(Q(4,\eta^{\epsilon})+\frac{1}{2}Q(3,\eta^{\epsilon})^2\right)+\cdots+{\epsilon}^{\frac{s-2}{2}}\ell(\epsilon,\eta^{\epsilon})+\Re(\epsilon,\eta^{\epsilon})\right\}
\end{align*}
where the coefficient $\ell(\epsilon,\eta^{\epsilon})$ of ${\epsilon}^{\frac{s-2}{2}}$ is a functional of $\eta^{\epsilon}$ and depends on $\epsilon.$ Special attention needs to be paid to the structure of the remainder term $\Re(\epsilon,\eta^{\epsilon}).$ There are two different aspects $\Re(\epsilon,\eta^{\epsilon})=\Re_1(\epsilon,\eta^{\epsilon})+\Re_2(\epsilon,\eta^{\epsilon}),$ where the first $\Re_1(\epsilon,\eta^{\epsilon})$ can be bounded by powers of $\eta^{\epsilon}$, while the second $\Re_2(\epsilon,\eta^{\epsilon})$ involves a part $e^{\theta(a)\cdot
a}.$ By taking conditions
(\ref{condition-on-F})-(\ref{condition-on-H}) into account,
$$|\ell(\epsilon,\eta^{\epsilon})|+|\Re_1(\epsilon,\eta^{\epsilon})|\leq \epsilon^{\frac{s-1}{2}}\cdot c\cdot\left(1+\left|\eta_T^{\epsilon}\right|^{k}+\int_0^T\left|\eta_t^{\epsilon}\right|^{k}dt\right),$$
for some nonnegative constants $c$ and $k.$ The second one $\Re_2(\epsilon,\eta^{\epsilon})$ on
the set $\{||\eta^{\epsilon}||<\epsilon^{-1/2}h\}$) can be estimated as
$$|\Re_2(\epsilon,\eta^{\epsilon})|\leq \epsilon^{\frac{s-1}{2}}\cdot c\cdot\left(1+\left|\eta_T^{\epsilon}\right|^{k}+\int_0^T\left|\eta_t^{\epsilon}\right|^{k}dt\right)e^{h\cdot||\eta^{\epsilon}||^2}.$$
H\"{o}lder's inequality and the fact
$\mathbb{E}^{z^{\epsilon}}1_{\{||\eta^{\epsilon}||<\epsilon^{-1/2}h\}}e^{h\cdot||\eta^{\epsilon}||^2}<C<\infty$
uniformly in $\epsilon$ (see Section 5.2.6 in
\cite{Wentzell-LD-Markov-Processes-1986}) suggest that we only need to
take care of $\Re_1(\epsilon,\eta^{\epsilon}).$

Taylor's expansion for $H(\phi_0+{\epsilon}^{1/2}\eta^{\epsilon})$ at
$\phi_0$ up to $(s-2)$-derivative gives
\begin{equation}\label{first-step-3}
\begin{aligned}
H(\phi_0+{\epsilon}^{1/2}\eta^{\epsilon})&=H(\phi_0)+{\epsilon}^{1/2}H^{(1)}(\phi_0)(\eta^{\epsilon})+\cdots+\frac{{\epsilon}^{\frac{s-2}{2}}}{(s-2)!}H^{(s-2)}(\phi_0)(\eta^{\epsilon},\cdots,\eta^{\epsilon})+IR_3,\\
&\qquad\qquad
IR_3=\epsilon^{\frac{s-1}{2}}\int_0^1\frac{(1-u)^{s-2}}{(s-2)!}H^{(s-1)}(\phi_0+u{\epsilon}^{1/2}\eta^{\epsilon})(\eta^{\epsilon},\cdots,\eta^{\epsilon})du,
\end{aligned}
\end{equation}
Now we combine (\ref{first-step-2}) and~(\ref{first-step-3}) to
rewrite (\ref{first-step-1.5}) on the set
$\{||\eta^{\epsilon}||<\epsilon^{-1/2}h\}$ as follows,
\begin{equation}\label{first-step-4}
\begin{aligned}
&\mathbb{E}^{\epsilon}\left[1_{\left\{||\eta^{\epsilon}||<\epsilon^{-1/2}h\right\}}H(\xi^{\epsilon})\exp\{{\epsilon}^{-1}F(\xi^{\epsilon})\}\right]\\
&=\exp\left\{{\epsilon}^{-1}\left[F(\phi_0)-S(\phi_0)\right]\right\}\times\\
&\quad
\mathbb{E}^{z^{\epsilon}}\Big\{\exp\{Q(2,\eta^{\epsilon})\}\cdot\Big(H(\phi_0)+{\epsilon}^{1/2}\left[Q(3,\eta^{\epsilon})H(\phi_0)+H^{(1)}(\phi_0)(\eta^{\epsilon})\right]\\
&\qquad+\epsilon\left[\left(Q(4,\eta^{\epsilon})+\frac{1}{2}[Q(3,\eta^{\epsilon})]^2\right)H(\phi_0)+H^{(2)}(\phi_0)(\eta^{\epsilon},\eta^{\epsilon})+Q(3,\eta^{\epsilon})H^{(1)}(\phi_0)(\eta^{\epsilon})\right]\\
&\qquad+\cdots+{\epsilon}^{\frac{s-2}{2}}\tilde{\ell}(\epsilon,\eta^{\epsilon})\Big)\Big\}\\
&\,\,\,\,\,\,-\exp\left\{{\epsilon}^{-1}\left[F(\phi_0)-S(\phi_0)\right]\right\}\times\\
&\quad\quad
\mathbb{E}^{z^{\epsilon}}\Big\{1_{\left\{||\eta^{\epsilon}||\geq \epsilon^{-1/2}h\right\}}\exp\{Q(2,\eta^{\epsilon})\}\cdot\Big(H(\phi_0)+{\epsilon}^{1/2}\left[Q(3,\eta^{\epsilon})H(\phi_0)+H^{(1)}(\phi_0)(\eta^{\epsilon})\right]\\
&\qquad+\epsilon\left[\left(Q(4,\eta^{\epsilon})+\frac{1}{2}[Q(3,\eta^{\epsilon})]^2\right)H(\phi_0)+H^{(2)}(\phi_0)(\eta^{\epsilon},\eta^{\epsilon})+Q(3,\eta^{\epsilon})H^{(1)}(\phi_0)(\eta^{\epsilon})\right]\\
&\qquad+\cdots+{\epsilon}^{\frac{s-2}{2}}\bar{\ell}(\epsilon,\eta^{\epsilon})\Big)\Big\}\\
&\,\,\,\,\,\,+\exp\left\{{\epsilon}^{-1}\left[F(\phi_0)-S(\phi_0)\right]\right\}\cdot
E^{z^{\epsilon}}\left\{1_{\left\{||\eta^{\epsilon}||<\epsilon^{-1/2}h\right\}}\exp\{Q(2,\eta^{\epsilon})\}\times\Re_1(\epsilon,\eta^{\epsilon})\right\}
\end{aligned}
\end{equation}
for two functionals $\tilde{\ell}(\epsilon,\eta^{\epsilon})$ and $\bar{\ell}(\epsilon,\eta^{\epsilon}).$ Now it becomes quite clear that Theorem \ref{main-thm} is proved if the following (I)-(III) are proved:

\textbf{(I)}.
\begin{align}\label{goal-I}
\mathbb{E}^{z^{\epsilon}}\left\{1_{\left\{||\eta^{\epsilon}||<\epsilon^{-1/2}h\right\}}\exp\{Q(2,\eta^{\epsilon})\}\times\Re_1(\epsilon,\eta^{\epsilon})\right\}=o\left({\epsilon}^{\frac{s-2}{2}}\right).
\end{align}

\textbf{(II)}. Each term over the set $\left\{||\eta^{\epsilon}||\geq
\epsilon^{-1/2}h\right\}$ in
(\ref{first-step-4}) tends to zero exponentially fast.

\textbf{(III)}. The expectations $\mathbb{E}^{z^{\epsilon}}$ in
(\ref{first-step-4}) without $\left\{||\eta^{\epsilon}||\geq
\epsilon^{-1/2}h\right\}$ have precise asymptotic expansions. Obviously, this is the place where normal deviations for stochastic processes are used. To make (III) more precise, we need to show the following asymptotic expansions for normal deviations
\begin{equation*}
\begin{aligned}
&\mathbb{E}^{z^{\epsilon}}\left\{\exp\{Q(2,\eta^{\epsilon})\}\right\}=\mathbb{E}\left\{\exp\{Q(2,\eta)\}\right\}+{\epsilon}^{1/2}C_{01}+\epsilon C_{02}+\cdots+o({\epsilon}^{\frac{s-2}{2}}),\\
&\mathbb{E}^{z^{\epsilon}}\left\{\exp\{Q(2,\eta^{\epsilon})\}\left[Q(3,\eta^{\epsilon})H(\phi_0)+H^{(1)}(\phi_0)(\eta^{\epsilon})\right]\right\}\\
&\quad=\mathbb{E}\left\{\exp\{Q(2,\eta)\}\left[Q(3,\eta)H(\phi_0)+H^{(1)}(\phi_0)(\eta)\right]\right\}+{\epsilon}^{1/2}C_{11}+\epsilon C_{12}+\cdots+o(\epsilon^{\frac{s-3}{2}}),\\
&\mathbb{E}^{z^{\epsilon}}\left\{\exp\{Q(2,\eta^{\epsilon})\}\left[\left(Q(4,\eta^{\epsilon})+\frac{1}{2}[Q(3,\eta^{\epsilon})]^2\right)H(\phi_0)+H^{(2)}(\phi_0)(\eta^{\epsilon},\eta^{\epsilon})+Q(3,\eta^{\epsilon})H^{(1)}(\phi_0)(\eta^{\epsilon})\right]\right\}\\
&\quad=\mathbb{E}\left\{\exp\{Q(2,\eta)\}\left[\left(Q(4,\eta)+\frac{1}{2}[Q(3,\eta)]^2\right)H(\phi_0)\right.\right.\\
&\qquad\left.\left.+H^{(2)}(\phi_0)(\eta,\eta)+Q(3,\eta)H^{(1)}(\phi_0)(\eta)\right]\right\}+{\epsilon}^{1/2}C_{21}+\epsilon C_{22}+\cdots+o(\epsilon^{\frac{s-4}{2}})
\end{aligned}
\end{equation*}
and so on, where $C_{ij}$ are constants which can be determined by
Theorem \ref{theorem-from-SPA} in Section \ref{subsec:normal-deviation}. If we replace all
terms in~(\ref{first-step-4}) by these asymptotic expansions for normal deviations, then we get
\begin{align*}
&\mathbb{E}^{\epsilon}\left[H(\xi^{\epsilon})\exp\{{\epsilon}^{-1}F(\xi^{\epsilon})\}\right]=\mathbb{E}^{\epsilon}\left[||\xi^{\epsilon}-\phi_0||<h;H(\xi^{\epsilon})\exp\{{\epsilon}^{-1}F(\xi^{\epsilon})\}\right]\\
&\qquad\qquad\qquad\qquad\qquad\qquad\qquad\qquad+o(\exp\{{\epsilon}^{-1}[F(\phi_0)-S(\phi_0)-\gamma]\})\\
&=\exp\left\{{\epsilon}^{-1}\left[F(\phi_0)-S(\phi_0)\right]\right\}\cdot\left[\sum_{0\leq
i\leq
(s-2)}K_i{\epsilon}^{i/2}+o\left({\epsilon}^{(s-2)/2}\right)\right],
\end{align*}
which is exactly (\ref{final-formula}). That is, Theorem \ref{main-thm} is proved
if (I), (II), (III) and the finiteness of $K_i$ are proved.

\subsection{Proofs of (I)-(III)}\label{subsec:proofs-I-III}
\subsubsection{Proof of (I)}\label{proof-of-i}
It is clear that (I) will be proved if for any positive integer $k,$
\begin{align}\label{goal-I-prime}
\mathbb{E}^{z^{\epsilon}}\left\{1_{\left\{||\eta^{\epsilon}||<\epsilon^{-1/2}h\right\}}\exp\{Q(2,\eta^{\epsilon})\}\epsilon^{\frac{s-1}{2}}\left(1+\left|\eta_T^{\epsilon}\right|^{k}+\int_0^T\left|\eta_t^{\epsilon}\right|^{k}dt\right)\right\}=o\left({\epsilon}^{\frac{s-2}{2}}\right).
\end{align}
In order to prove (\ref{goal-I-prime}), we establish a lemma first.
\begin{lemma}\label{good-lemma}
Under the assumption (F), for any positive $k,$ there are
constants $c_1$ and $c_2$ (depending only on
$k$) such that for all $t\in[0,T],$ $1>{\epsilon}>0,$
\begin{align*}
 \mathbb{E}^{z^{\epsilon}}\left[\left(\eta^{\epsilon}_t\right)^{k}\right]\leq
 t\cdot c_1+c_1\cdot c_2\cdot\int_0^ts\cdot
 e^{c_2\cdot(t-s)}ds.
\end{align*}
\end{lemma}
\begin{proof}
We consider a sequence of functions
$f_n(x)=\frac{x^k}{1+\left(x/n\right)^k}$ with an even positive
integer $k.$ Then, according to (\ref{compen-eta-h}), the generating
operator $A_t^{\eta^{\epsilon}}$ applying to $f_n$ gives
\begin{align*}
A_t^{\eta^{\epsilon}}f_n(x)&=\epsilon^{-1/2}\left[\frac{\partial
G_0}{\partial
z}(t,\phi_0(t)+\epsilon^{1/2}x;z_0(t))-\phi_0'(t)\right]f_n'(x)+\frac{1}{2}a(t,\phi_0(t)+\epsilon^{1/2}x)f''_n(x)\\
&\qquad+
{\epsilon}^{-1}\int_{\mathbb{R}}\left[f_n(x+\epsilon^{1/2}u)-f_n(x)-\epsilon^{1/2}f_n'(x)\cdot
u\right]e^{z_0(t)u}\nu_{t,\phi_0(t)+\epsilon^{1/2}x}(du)\\
&=\frac{\partial^2 G_0}{\partial z\partial
x}(t,\phi_0(t)+\theta_1\epsilon^{1/2}x;z_0(t))\cdot x\cdot f_n'(x)\\
&\qquad+\frac{f_n''(x)}{2}\left(\int_{\mathbb{R}}u^2e^{z_0(t)u}\nu_{t,\phi_0(t)+\epsilon^{1/2}x}(du)+a(t,\phi_0(t)+\epsilon^{1/2}x)\right)\\
&\qquad+\cdots+\frac{1}{k!}\int_{\mathbb{R}}\epsilon^{\frac{k-2}{2}}u^kf_n^{(k)}(x+\theta_2\epsilon^{1/2}u)e^{z_0(t)u}\nu_{t,\phi_0(t)+\epsilon^{1/2}x}(du),
\end{align*}
which is less than or equal to $c_3+c_4\cdot f_n(x),$ since $f_n^{(k)}$ is bounded, and $x\cdot
f'_n(x),f_n'',f_n''',\cdots,f_n^{(k-1)}$ are bounded by
$c_5+c_6\cdot f_n(x).$ So
\begin{align*}
\mathbb{E}^{z^{\epsilon}}\left[f_n\left(\eta^{\epsilon}_t\right)\right]\leq
f_n(0)+\int_0^t\left(c_1+c_2\cdot
\mathbb{E}^{z^{\epsilon}}\left[f_n\left(\eta^{\epsilon}_s\right)\right]\right)ds.
\end{align*}
Applying Gronwall's lemma with such nonnegative
$\mathbb{E}^{z^{\epsilon}}\left[f_n\left(\eta^{\epsilon}_t\right)\right],$ we
obtain
\begin{align*}
\mathbb{E}^{z^{\epsilon}}\left[f_n\left(\eta^{\epsilon}_t\right)\right]\leq
t\cdot c_1+c_1\cdot c_2\cdot\int_0^ts\cdot
e^{c_2\cdot(t-s)}ds,
\end{align*}
then the proof is done by sending $n$ to infinity.
\end{proof}
From this lemma, we have $\sup_{\epsilon\in(0,1)}
\mathbb{E}^{z^{\epsilon}}\left[(\eta^{\epsilon}_T)^{k}\right]<\infty$ and
$\sup_{\epsilon\in(0,1)}
\mathbb{E}^{z^{\epsilon}}\left[\int_0^T(\eta^{\epsilon}_t)^{k}dt\right]<\infty$
for any integer $k.$ These together with the fact
$\exp\{Q(2,\eta^{\epsilon})\}\leq1$ yield (\ref{goal-I-prime}).

\subsubsection{Proof of (II)}\label{proof-of-ii} It is immediate that
the first term goes to zero exponentially fast,
\begin{align*}
\mathbb{E}^{z^{\epsilon}}\left\{1_{\left\{||\eta^{\epsilon}||\geq
\epsilon^{-1/2}h\right\}}\exp\{Q(2,\eta^{\epsilon})\}\cdot
H(\phi_0)\right\}\rightarrow0\text{ exponentially fast},
\end{align*}
which is from (\ref{eponential-conv-zero}). Every other term has an upper bound by using estimates
(\ref{condition-on-F})-(\ref{condition-on-H}):
\begin{equation}\label{uniform-bound-expectation}
c\cdot
\mathbb{E}^{z^{\epsilon}}\left\{1_{\left\{||\eta^{\epsilon}||\geq
\epsilon^{-1/2}h\right\}}\epsilon^{\frac{j-2}{2}}\left(1+\left|\eta^{\epsilon}_T\right|^{k}+\int_0^T\left|\eta^{\epsilon}_t\right|^{k}dt\right)\right\},
\end{equation}
for some nonnegative constants $c$ and $k.$ By applying H\"{o}lder's
inequality to (\ref{uniform-bound-expectation}) and using Lemma
\ref{good-lemma}, it follows each term over the set $\left\{||\eta^{\epsilon}||\geq
\epsilon^{-1/2}h\right\}$ in
(\ref{first-step-4}) tends to zero exponentially fast.

\subsubsection{Proof of (III)}\label{proof-of-iii}
\textbf{(a) Proof of the first expansion}\\
The task is to prove
$$\mathbb{E}^{z^{\epsilon}}\left\{\exp\{Q(2,\eta^{\epsilon})\}\right\}=\mathbb{E}\left\{\exp\{Q(2,\eta)\}\right\}+{\epsilon}^{1/2}C_{01}+\epsilon C_{02}+\cdots+{\epsilon}^{(s-2)/2}C_{0(s-s)}+o({\epsilon}^{(s-2/2}).$$
Denoting
\begin{align*}
\widetilde{F}(x[0,T])=\exp\left\{Q(2,x[0,T])\right\}=&\exp\left\{\frac{1}{2}F^{(2)}(\phi_0)(x[0,T],x[0,T])\right.\\
&\left.+\int_0^T\frac{1}{2}(x(t))^2\frac{\partial^2
G_0}{\partial x^2}(t,\phi_0(t);z_0(t))dt\right\},
\end{align*} we need to check
such $\widetilde{F}(x[0,T])$ satisfies all conditions of Theorem \ref{theorem-from-SPA} in Section \ref{subsec:normal-deviation}.

Claim 1. $\widetilde{F}(x[0,T])$ is $3(s-2)$ times differentiable. It is easy to see that $\widetilde{F}$ is infinitely differentiable. Furthermore, the derivatives can be computed as follows
\begin{align*}
\widetilde{F}^{(1)}&(x[0,T])(\delta)=\lim_{h\rightarrow0}{h}^{-1}\left[\widetilde{F}(x[0,T]+h\delta)-\widetilde{F}(x[0,T])\right]\\
&=\widetilde{F}(x[0,T])\left(F^{(2)}(\phi_0)(x[0,T],\delta)+\int_0^Tx(t)\delta(t)\frac{\partial^2
G_0}{\partial x^2}(t,\phi_0(t);z_0(t))dt\right),\\
\widetilde{F}^{(2)}&(x[0,T])(\delta_1,\delta_2)=\lim_{h\rightarrow0}h^{-1}\left[\widetilde{F}^{(1)}(x[0,T]+h\delta_2)(\delta_1)-\widetilde{F}^{(1)}(x[0,T])(\delta_1)\right]\\
&=\widetilde{F}(x[0,T])\left(F^{(2)}(\phi_0)(x[0,T],\delta_2)+\int_0^Tx(t)\delta_2(t)\frac{\partial^2
G_0}{\partial x^2}(t,\phi_0(t);z_0(t))dt\right)\\
&\qquad\qquad\qquad\cdot\left(F^{(2)}(\phi_0)(x[0,T],\delta_1)+\int_0^Tx(t)\delta_1(t)\frac{\partial^2
G_0}{\partial x^2}(t,\phi_0(t);z_0(t))dt\right)\\
&\,\,+\widetilde{F}(x[0,T])\left(F^{(2)}(\phi_0)(\delta_1,\delta_2)+\int_0^T\delta_1(t)\delta_2(t)\frac{\partial^2
G_0}{\partial x^2}(t,\phi_0(t);z_0(t))dt\right)
\end{align*}
and so on.

Claim 2. $\widetilde{F}$ satisfies condition (I) of Theorem \ref{theorem-from-SPA}. First we have
$\left|\widetilde{F}(x[0,T])\right|\leq1.$ The derivatives of
$\widetilde{F}$ satisfy (I). For instance, here we check for
$\widetilde{F}^{(2)}.$
\begin{align*}
&\left|\widetilde{F}^{(2)}(x[0,T])(\delta,\delta)\right|\leq
\left(F^{(2)}(\phi_0)(x[0,T],\delta)+\int_0^Tx(t)\delta(t)\frac{\partial^2
G_0}{\partial x^2}(t,\phi_0(t);z_0(t))dt\right)^2+c(||\delta||),
\end{align*}
where $c(||\delta||)$ is a constant depending on the uniform norm of
$\delta.$ Taking into account the assumptions on $F^{(2)},$ the above is less
than or equal to
\begin{align*}
p^2\cdot\left(|x(T)\delta(T)|^m+(1+T||\delta||^n)(1+\int_0^T|x(t)|^{n}dt)+||\delta\frac{\partial^2
G_0}{\partial x^2}||\int_0^Tx(t)^2dt\right)^2+c(||\delta||).
\end{align*}
We apply H\"{o}lder
inequality several times to get an upper bound
$$c_1(||\delta||)\left(1+|x(T)|^{2m}+\int_0^T|x(t)|^{2n}dt\right).$$

Claim 3.
$\widetilde{F}^{(i)}(x[0,T])(I_{[t,T]}\delta,\cdots,I_{[t,T]}\delta),3\leq
i\leq3(s-2),$ are continuous with respect to $x[0,T]$ uniformly as
$x[0,T]$ changes over an arbitrary compact subset of $D_0[0,T],$ $t$
over $[0,T],$ and $\delta[0,T]$ over the set of Lipschitz continuous
functions with constant $1,$ $||\delta||\leq 1.$

It follows from~\cite{Skorohod-Limit-Theorem-1956} that any compact
subset of $D_0[0,T]$ is a bounded set in uniform topology. Taking
$|I_{[t,T]}\delta|\leq 1$ into account, Claim 3 is done easily.
For instance, the uniform continuity of
$F^{(2)}(\phi_0)(x[0,T],I_{[t,T]}\delta)$ in $x[0,T]$ can be
achieved as follows:
\begin{align*}
&\left|F^{(2)}(\phi_0)(y[0,T],I_{[t,T]}\delta)-F^{(2)}(\phi_0)(x[0,T],I_{[t,T]}\delta)\right|\\
&=\left|F^{(2)}(\phi_0)(y[0,T]-x[0,T],I_{[t,T]}\delta)\right|\leq
C||y-x||,
\end{align*}
where $C$ is independent of $x,y,t,\delta.$ Uniform continuity of
$F^{(2)}(\phi_0)(x[0,T],x[0,T])$ in $x[0,T]$ can be also obtained by
\begin{align*}
&\left|F^{(2)}(\phi_0)(y[0,T],y[0,T])-F^{(2)}(\phi_0)(x[0,T],x[0,T])\right|\\
&\leq
\left|F^{(2)}(\phi_0)(y[0,T]-x[0,T],y[0,T])\right|+\left|F^{(2)}(\phi_0)(y[0,T]-x[0,T],x[0,T])\right|\\
&\leq C\cdot\left(||y||+||x||\right)\cdot||x-y||\leq C\cdot
C_1\cdot||x-y||,
\end{align*}
where $C$ is independent of $x,y,$ and $C_1$ can be chosen
independently of $x,y,$ because a compact set is a bounded set in
uniform topology.\\
\textbf{(b) Proofs of the second expansion and the other expansions}\\
The second expansion is
\begin{align*}
&\mathbb{E}^{z^{\epsilon}}\left\{\exp\{Q(2,\eta^{\epsilon})\}\left[Q(3,\eta^{\epsilon})H(\phi_0)+H^{(1)}(\phi_0)(\eta^{\epsilon})\right]\right\}\\
&\quad=\mathbb{E}\left\{\exp\{Q(2,\eta)\}\left[Q(3,\eta)H(\phi_0)+H^{(1)}(\phi_0)(\eta)\right]\right\}+{\epsilon}^{1/2}C_{11}+\epsilon
C_{12}+\cdots+o(\epsilon^{(s-3)/2}).
\end{align*}
We use $\widehat{F}$ to denote
\begin{align*}
\widehat{F}(x[0,T])&=\exp\left\{Q(2,x[0,T])\right\}[Q(3,x[0,T])+H^{(1)}(\phi_0)(x[0,T])]\\
&=\widetilde{F}(x[0,T])\cdot [Q(3,x[0,T])+H^{(1)}(\phi_0)(x[0,T])].
\end{align*}

Claim 1. $\widehat{F}(x[0,T])$ is infinitely differentiable. The first two derivatives are
\begin{align*}
&\widehat{F}^{(1)}(x[0,T])(\delta)=\lim_{{h}\rightarrow0}{h}^{-1}\left[\widehat{F}(x[0,T]+h\delta)-\widehat{F}(x[0,T])\right]\\
&=\widetilde{F}^{(1)}(x[0,T])(\delta)\cdot[Q(3,x[0,T])+H^{(1)}(\phi_0)(x[0,T])]+\frac{3}{3!}\widetilde{F}(x[0,T])\\
&\quad\times\left(F^{(3)}(\phi_0)(x[0,T],x[0,T],\delta)+\int_0^Tx^2(t)\delta(t)\frac{\partial^3
G_0}{\partial
x^3}(t,\phi_0(t);z_0(t))dt+H^{(1)}(\phi_0)(\delta)\right),
\end{align*}
\begin{align*}
&\widehat{F}^{(2)}(x[0,T])(\delta_1,\delta_2)=\lim_{{h}\rightarrow0}{h}^{-1}\left[\widehat{F}^{(1)}(x[0,T]+h\delta_2)(\delta_1)-\widehat{F}^{(1)}(x[0,T])(\delta_1)\right]\\
&=\widetilde{F}^{(2)}(x[0,T])(\delta_1,\delta_2)\cdot[Q(3,x[0,T])+H^{(1)}(\phi_0)(x[0,T])]\\
&+\frac{3}{3!}\sum_{1\leq i\neq j\leq
2}\widetilde{F}^{(1)}(x[0,T])(\delta_j)\Big(F^{(3)}(\phi_0)(x[0,T],x[0,T],\delta_i)\\
&\qquad\qquad\qquad\qquad\qquad+\int_0^Tx^2(t)\delta_i(t)\frac{\partial^3
G_0}{\partial x^3}(t,\phi_0(t);z_0(t))dt+H'(\phi_0)(\delta_i)\Big)\\
&+\frac{6}{3!}\widetilde{F}(x[0,T])\left(F^{(3)}(\phi_0)(x[0,T],\delta_1,\delta_2)+\int_0^Tx(t)\delta_1(t)\delta_2(t)\frac{\partial^3
G_0}{\partial x^3}(t,\phi_0(t);z_0(t))dt\right).
\end{align*}

Claim 2. $\widehat{F}$ satisfies condition (I) of Theorem \ref{theorem-from-SPA}. We notice that
$Q(3,x[0,T])$ satisfies
\begin{align*}
&|Q(3,x[0,T])|=\left|\frac{1}{3!}F^{(3)}(\phi_0)(x[0,T],x[0,T],x[0,T])+\int_0^T\frac{1}{3!}x^3(t)\frac{\partial^3
G_0}{\partial x^3}(t,\phi_0(t);z_0(t))dt\right|\\
&\leq\frac{p}{3!}\left[|x(T)|^{3m}+\left(1+\int_0^T|x(t)|^{n}dt\right)^3\right]+\frac{1}{3!}\left|\left|\frac{\partial^3
G_0}{\partial x^3}\right|\right|\cdot\int_0^T|x(t)|^3dt.
\end{align*}
This together with~(\ref{condition-on-H}) imply that the first part
of (I) is fulfilled. For the second part of (I) on derivatives
of $\widehat{F},$ similar arguments can be applied. We can also
prove the following

Claim 3.
$\widehat{F}^{(i)}(x[0,T])(I_{[t,T]}\delta,\cdots,I_{[t,T]}\delta),3\leq
i\leq 3(s-3),$ are continuous with respect to $x[0,T]$ uniformly as
$x[0,T]$ changes over an arbitrary compact subset of $D_0[0,T],$ $t$
over $[0,T],$ and $\delta[0,T]$ over the set of Lipschitz continuous
functions with constant $1$ and $||\delta||\leq 1.$

The other expansions are proved in the same way.

\subsection{Finiteness of the coefficients}\label{subsec:finiteness-coefficient}
First it is obvious that $K_0$ is finite since $H$ is bounded and
$\exp\{Q(2,\eta)\}\leq1.$ For the rest, we will use Theorem 2.3.1
in \cite{Wentzell-LD-Markov-Processes-1986} to prove each finiteness.
The very first requirement of Theorem 2.3.1 is that
$G^{\eta}(t,x;z)$ satisfies condition (A) in Section \ref{sec:main-theorem}. Let us recall $G^{\eta}(t,x;z):$
\begin{align*}
G^{\eta}(t,x;z)&=zx\cdot\frac{\partial^2 G_0}{\partial z
\partial x}(t,\phi_0(t);z_0(t))+\frac{1}{2}z^2\cdot\frac{\partial^2 G_0}{\partial
z^2}(t,\phi_0(t);z_0(t)),
\end{align*}
which doesn't satisfies condition (A) obviously because of the
linear term in $x.$ So we turn to consider the following
transformation
$$\widetilde{\eta}_t=\exp\left\{-\int_0^tg(s)ds\right\}\eta_t,$$
where $g(t)=\frac{\partial^2 G_0}{\partial z\partial
x}(t,\phi_0(t);z_0(t)).$ Straightforward computation will give us
\begin{align*}
\mathfrak{A}^{\widetilde{\eta}}f(t,x)&=\frac{\partial f}{\partial
t}(t,x)+\frac{1}{2}\frac{\partial^2 G_0}{\partial
z^2}(t,\phi_0(t);z_0(t))\cdot\frac{\partial^2 f}{\partial
x^2}(t,x)\exp\left\{-2\int_0^tg(s)ds\right\},\\
G^{\widetilde{\eta}}(t,x;z)&=\frac{1}{2}z^2\cdot\frac{\partial^2
G_0}{\partial
z^2}(t,\phi_0(t);z_0(t))\exp\left\{-2\int_0^tg(s)ds\right\},\\
H^{\widetilde{\eta}}(t,x;u)&=\frac{1}{2}\frac{\partial^2
H_0}{\partial
u^2}(t,\phi_0(t);\phi_0'(t))\exp\left\{2\int_0^tg(s)ds\right\}u^2.
\end{align*}
Now $G^{\widetilde{\eta}}(t,x;z)$ satisfies condition (A). We will
apply Theorem 2.3.1 restricting ourself to
$G^{\widetilde{\eta}}(t,x;z).$ It is then shown that each finiteness
can be deduced from this.
\subsubsection{Finiteness of
$K_1$}\label{finiteness-k-1} We now show an auxiliary result.
\begin{lemma}\label{auxiliary-result}
For any positive integer $j,$
\begin{align}\label{auxiliary-estimate}
\mathbb{E}(||\widetilde{\eta}||^j)<\infty.
\end{align}
\end{lemma}
\begin{proof}
The normalized action functional for the family of processes
$\sqrt{\epsilon}\cdot\widetilde{\eta}$ is
\begin{align*}
I^{\widetilde{\eta}}(f(\cdot))=\int_0^TH^{\widetilde{\eta}}(t,f(t);f'(t))dt.
\end{align*}
Let us consider, for some positive $\alpha,$ positive integer $m,$
\begin{align*}
&\Phi_0^{\widetilde{\eta}}(m)=\left\{f\in
D_0[0,T]:I^{\widetilde{\eta}}(f(\cdot))\leq
m\right\};\\
&\Phi_0^{\widetilde{\eta}}(m)_{+\alpha\sqrt{m}}:
\alpha\sqrt{m}-\text{neighbourhood of }\Phi_0^{\widetilde{\eta}}(m).
\end{align*}
The space $D_0[0,T]$ decomposes into the union
$$\Phi_0^{\widetilde{\eta}}(1)_{+\alpha}\cup\bigcup_{m=1}^{\infty}\Phi_0^{\widetilde{\eta}}(m+1)_{+\alpha\sqrt{m+1}}\backslash\Phi_0^{\widetilde{\eta}}(m)_{+\alpha\sqrt{m}},$$
thus we have
\begin{equation}\label{oct-16-2}
\begin{aligned}
\mathbb{E}(||\widetilde{\eta}||^j)&\leq \sum_{m=0}^{\infty} \mathbb{E}\left\{\widetilde{\eta}\in \Phi_0^{\widetilde{\eta}}(m+1)_{+\alpha\sqrt{m+1}}\backslash\Phi_0^{\widetilde{\eta}}(m)_{+\alpha\sqrt{m}};||\widetilde{\eta}||^j\right\}\\
&\leq \sum_{m=0}^{\infty}
\mathbb{P}\left\{\widetilde{\eta}\notin\Phi_0^{\widetilde{\eta}}(m)_{+\alpha\sqrt{m}}\right\}\cdot\sup\left\{||f||^j:
f\in\Phi_0^{\widetilde{\eta}}(m+1)_{+\alpha\sqrt{m+1}}\right\}.
\end{aligned}
\end{equation}
We first analyze the supremum term in~(\ref{oct-16-2}). Let us recall
a fact that, for any fixed $a>0,$ any integer $l>0,$ there exists a
constant $A>0$ (only depends on $a$ and $l,$ independent of $x$)
such that
\begin{align}\label{importnat-estimate-exponential}
|x|^l\leq Ae^{ax^2},\quad\text{ for all }x\in \mathbb{R}.
\end{align}
Thus, for any $a>0,$ any positive integer $j,$ there is some
$A=A(a,j)>0$ such that
\begin{align*}
||f||^j\leq A\exp\left\{a||f||^2\right\}.
\end{align*}
And for any
$f\in\Phi_0^{\widetilde{\eta}}(m+1)_{+\alpha\sqrt{m+1}},$ we can
choose a small $a$ such that (such $a$ can be chosen independent of
$m$ by using Lemma 5.2.5 on in \cite{Wentzell-LD-Markov-Processes-1986})
$$a||f||^2\leq \frac{1}{3}(m+1).$$
So the supremum term can be estimated as follows
\begin{align}\label{good-1}
\sup\left\{||f||^j:
f\in\Phi_0^{\widetilde{\eta}}(m+1)_{+\alpha\sqrt{m+1}}\right\}\leq A
\exp\left\{\frac{1}{3}\cdot(m+1)\right\}.
\end{align}

Now we analyze the probabilities in~(\ref{oct-16-2}) by using
Theorem 2.3.1 in \cite{Wentzell-LD-Markov-Processes-1986}. We check all the
conditions of Theorem 2.3.1 as follows. We suppose
$\sup_t|\frac{\partial^2 G_0}{\partial
z^2}(t,\phi_0(t);z_0(t))\exp\left\{-2\int_0^tg(s)ds\right\}|\leq c$
for some constant $c.$ Let us consider a constant $Z,$ an integer
$n,$ whose values will be determined a little later. We set
$\epsilon_2=\frac{mZ^2c^2\kappa}{6T}, t_i=\frac{iT}{n},\triangle
t_{\min}=\triangle t_{\max}=\frac{T}{n},k=2,$
$z(1)=\sqrt{m}Z,z(2)=-\sqrt{m}Z,$ $d(1)=d(2)=mZ^2c,$
$\delta'=\frac{\alpha\sqrt{m}}{3}, A=m,$
$$U_0=\left\{u:z(j)\cdot u<d(j),j=1,2\right\}=\left(-\sqrt{m}Zc,\sqrt{m}Zc\right).$$
Now we define a small $\epsilon_1$ such that
$\epsilon_1(2-\epsilon_1)+T\epsilon_1(3-\epsilon_1)+\frac{Z^2c^2\kappa}{6}(1-\epsilon_1)\leq
\frac{Z^2c^2\kappa}{3}.$ Firstly we know $G^{\widetilde{\eta}}$
satisfies condition $A$ with $\overline{G}(z)=\frac{1}{2}z^2c.$
Secondly $G^{\widetilde{\eta}}$ has the property:
\begin{align*}
G^{\widetilde{\eta}}(t,y;(1-\epsilon_1)z)\leq
(1-\epsilon_1)G^{\widetilde{\eta}}(s,x;z),
\end{align*}
for $t,s$ which are close enough (this can be guaranteed by choosing
a large $n$). Finally we can approximate $H^{\widetilde{\eta}}$ on
$U_0$ by tangent lines from below
with any accuracy. More precisely, we can obviously find some
$z_0\{1\},\cdots,z_0\{N\}$ such that
\begin{align*}
&\sup_{u\in [-1,1]}\left(H^{\widetilde{\eta}}(t,x;u)-\max_{1\leq
j\leq
N}\left[z_0\{j\}u-G^{\widetilde{\eta}}(t,x;z_0\{j\})\right]\right)\leq
\frac{\kappa}{6T}.
\end{align*}
For general $u\in U_0,$ we set $z\{j\}=\sqrt{m}Zcz_0\{j\},1\leq
j\leq N$ (here $N$ can be chosen independent of $m$). For short, we will use
$h_1(t)=\frac{\partial^2H_0}{\partial
u^2}(t,\phi_0(t);\phi_0'(t))\exp\{2\int_0^tg(s)ds\}$ and
$h_2(t)=\frac{\partial^2G_0}{\partial
z^2}(t,\phi_0(t);z_0(t))\exp\{-2\int_0^tg(s)ds\}.$ Then
\begin{align*}
&\sup_{u\in U_0}\left(H^{\widetilde{\eta}}(t,x;u)-\max_{1\leq j\leq
N}\left[z\{j\}u-G^{\widetilde{\eta}}(t,x;z\{j\})\right]\right)\\
&=\sup_{u\in
\left(-\sqrt{m}Zc,\sqrt{m}Zc\right)}\left(\frac{1}{2}h_1(t)u^2-\max_{1\leq
j\leq
N}\left[z\{j\}u-\frac{1}{2}z\{j\}^2h_2(t)\right]\right)\\
&=mZ^2c^2\cdot\sup_{u\in
\left(-\sqrt{m}Zc,\sqrt{m}Zc\right)}\Big(\frac{1}{2}h_1(t)\left(\frac{u}{\sqrt{m}Zc}\right)^2-\\
&\qquad\qquad\qquad\qquad\qquad\qquad\qquad-\max_{1\leq j\leq
N}\left[\frac{z\{j\}}{\sqrt{m}Zc}\frac{u}{\sqrt{m}Zc}-\frac{1}{2}\left(\frac{z\{j\}}{\sqrt{m}Zc}\right)^2h_2(t)\right]\Big)\\
&= mZ^2c^2\cdot\sup_{u\in
(-1,1)}\left(\frac{1}{2}h_1(t)u^2-\max_{1\leq j\leq
N}\left[z_0\{j\}u-\frac{1}{2}z_0\{j\}^2h_2(t)\right]\right)\\
&\leq \frac{mZ^2c^2\cdot\kappa}{6T}=\epsilon_2.
\end{align*}
All conditions of Theorem 2.3.1 are thus checked. Applying this
theorem with $\delta'\geq \frac{T}{n}\sqrt{m}Zc,$ i.e.
$n\geq\frac{3ZTc}{\alpha},$ we get
\begin{align*}
&\mathbb{P}\left\{\widetilde{\eta}\notin\Phi_0^{\widetilde{\eta}}(m)_{+\alpha\sqrt{m}}\right\}=\mathbb{P}\left\{\text{dist}\left(\eta^{\epsilon},\Phi_0^{\widetilde{\eta}}(m)\right)\geq\alpha\sqrt{m}\right\}\\
&\leq4n\exp\left\{\frac{T}{n}\left[\frac{Z^2c m}{2}-Z^2c m\right]\right\}\\
&\qquad\qquad+N^n\exp\left\{-m+m\epsilon_1(2-\epsilon_1)+T\left(m\epsilon_1(3-\epsilon_1)+\frac{mZ^2c^2\kappa}{6T}(1-\epsilon_1)\right)\right\}\\
&=4n\exp\left\{-m\left[\frac{TZ^2c}{2n}\right]\right\}\\
&\qquad\qquad+N^n\exp\left\{-m\left[1-\epsilon_1(2-\epsilon_1)-T\epsilon_1(3-\epsilon_1)-\frac{Z^2c^2\kappa}{6}(1-\epsilon_1)\right]\right\}\\
&\leq4n\exp\left\{-m\left[\frac{TZ^2c}{2n}\right]\right\}+N^n\exp\left\{-m\left(1-\frac{Z^2c^2\kappa}{3}\right)\right\},\text{
definition of
}\epsilon_1\\
&=4n\exp\left\{-m\right\}+N^n\exp\left\{-m\left(1-\frac{Z^2c^2\kappa}{3}\right)\right\},\text{
choose }Z=\sqrt{\frac{2n}{Tc}}, n\geq \frac{18T}{\alpha^2c}.
\end{align*}
Noticing that $\frac{Z^2c^2\kappa}{3}$ can be as small as possible
by choosing a small $\kappa,$ we thus assume
$1-\frac{Z^2c^2\kappa}{3}>1/2.$ Then it follows
$$\mathbb{P}\left\{\widetilde{\eta}\notin\Phi_0^{\widetilde{\eta}}(m)_{+\alpha\sqrt{m}}\right\}\leq4n\exp\left\{-m\right\}+N^n\exp\left\{-\frac{m}{2}\right\}.$$
Let us now go back to~(\ref{oct-16-2}) combining above estimate and
(\ref{good-1})
\begin{align*}
\mathbb{E}(||\widetilde{\eta}||^j)&\leq
\sum_{m=0}^{\infty}\left(4n\exp\left\{-m\right\}+N^n\exp\left\{-\frac{m}{2}\right\}\right)A
\exp\left\{\frac{1}{3}\cdot(m+1)\right\}\\
&\leq
Ae^{1/3}(4n+N^n)\sum_{m=0}^{\infty}\exp\left\{-\frac{m}{6}\right\}<\infty.
\end{align*}
\end{proof}
By observing the transformation
$\widetilde{\eta}_t=\exp\left\{-\int_0^tg(s)ds\right\}\eta_t,$ we
immediately derive
\begin{align*}
\mathbb{E}(||\eta||^j)\leq
Be^{1/3}(4n+N^n)\sum_{m=0}^{\infty}\exp\left\{-\frac{m}{6}\right\}<\infty.
\end{align*}
It is clear that $K_1$ can be bounded by expectation of
$c_1+c_2\cdot||\eta||^j$ for some $j, c_1$ and $c_2,$ from
which finiteness of $K_1$ follows according to Lemma
\ref{auxiliary-result}.
\subsubsection{Finiteness of the rest of the
coefficients} Since all derivatives of $F$ and $H$ are bounded
symmetric linear functionals, we can use Lemma
\ref{auxiliary-result} to prove the finiteness of the rest of the
coefficients.

\section{Connections with partial integro-differential equations}\label{sec:Math-Physics}
The connections are between large (or normal) deviations and solutions to
\begin{equation}\label{eq:Math-Physics-general}
\begin{cases}
\begin{aligned}
\frac{\partial}{\partial t}u^{\epsilon}(t,x)=&\frac{\epsilon}{2}a(t,x)\Delta u^{\epsilon}(t,x)+b(t,x)\nabla u^{\epsilon}(t,x)+\epsilon^{-1}c(x)u^{\epsilon}(t,x)\\
&+\epsilon^{-1}\int_{\mathbb{R}}\left[u^{\epsilon}(t,x+\epsilon u)-u^{\epsilon}(t,x)-\epsilon u\nabla u^{\epsilon}(t,x)\right]\nu_{t,x}(du)
\end{aligned}\\
u^{\epsilon}(0,x)=g(x)
\end{cases}
\end{equation}
over $(t,x)\in \mathbb{R}_{+}\times \mathbb{R}.$ More precisely, it is expected that
\begin{equation}\label{eq:connections}
u^{\epsilon}(t,x)=\mathbb{E}^{\epsilon}_{0,x}\left[g(\xi^{\epsilon}_t)\exp\left\{\epsilon^{-1}\int_0^tc(\xi^{\epsilon}_s)ds\right\}\right],
\end{equation}
then the precise asymptotics for large deviations (or normal deviations) developed in Theorem \ref{main-thm} (or Theorem \ref{theorem-from-SPA}) can be applied. Of course, formula (\ref{eq:connections}) is not always true unless suitable conditions are imposed. In the first part of this section, we prove (\ref{eq:connections}) for a special case when $c(x)$ is a constant. Then from Theorem \ref{main-thm}, it follows
$$u^{\epsilon}(t,x)=e^{t/{\epsilon}}\cdot\left[\sum_{k=0}^nk_i(x) \epsilon^{k/2}+o(\epsilon^{n/2})\right].$$
The second part of this section is on the study of precise asymptotics of $u^{\epsilon}(t,x)$ in more general settings.

\subsection{The specific case}
In (\ref{eq:Math-Physics-general}), we set $c=1$ and $\nu_{t,x}(du)=u^21_{\{|u|\leq1\}}(du).$ What is more, we assume $a(t,x)=a(x),$ $b(t,x)=b(x),$ and $0<\inf_{x}a(x)\leq \sup_{x}a(x)<\infty,$ the smooth functions $a(x), b(x)$ and $g(x)$ are bounded together with their derivatives $d^j a/d x^j,$ $d^j b/d x^j$ and $d^j g/d x^j.$ In this case, we consider a family of jump processes $\xi^{\epsilon}$ with generating operators
$$A^{\epsilon}f(x)=\frac{\epsilon}{2}a(x)f''(x)+b(x)f'(x)+\epsilon^{-1}\int_{-1}^1\left[f(x+\epsilon u)-f(x)\right]u^2du$$
for continuous bounded $f$ together with its first and second derivatives. From the theory of semigroups, the function $$v^{\epsilon}(t,x):=\mathbb{E}^{\epsilon}_{0,x}f(\xi^{\epsilon}_t)$$ is the unique solution to the problem, for $f$ in the domain of $A^{\epsilon},$
\begin{equation*}
\begin{cases}
\begin{aligned}
\frac{\partial}{\partial t}v^{\epsilon}(t,x)=A^{\epsilon}v^{\epsilon}(t,x),
\end{aligned}\\
v^{\epsilon}(0,x)=f(x).
\end{cases}
\end{equation*}
Now it is easy to see that $u^{\epsilon}(t,x):=e^{t/{\epsilon}}\cdot\mathbb{E}^{\epsilon}_{0,x}f(\xi^{\epsilon}_t)$ is the unique solution of
\begin{equation*}
\begin{cases}
\begin{aligned}
\frac{\partial}{\partial t}u^{\epsilon}(t,x)=A^{\epsilon}u^{\epsilon}(t,x)+\epsilon^{-1}u^{\epsilon}(t,x),
\end{aligned}\\
u^{\epsilon}(0,x)=f(x).
\end{cases}
\end{equation*}
The conditions imposed on $a,b$ and $g$ are mainly for the smooth and growth assumptions in Theorem \ref{main-thm}, such as (F) and (\ref{condition-on-H}). The condition $c(x)=1$ is crucial in this special case since it forces the $\max[F-S]$ is reached uniquely at $\phi_0\equiv0.$
This example should be considered as the asymptotics for normal (not large) deviations since the main part of the integral (\ref{eq:connections}) is due to the most probable sample path (which is identically zero). Asymptotics for large (not normal) deviations can be seen below.

\subsection{In more general settings}
Let $\xi^{\epsilon}$ now be the locally infinitely divisible family of processes considered in Theorem \ref{main-thm} satisfying all the assumptions. Then the corresponding partial integro-differential equation is (\ref{eq:Math-Physics-general}) with $b(t,x)$ replaced by $\alpha(t,x).$ In order to show (\ref{eq:connections}), it is natural to impose suitable conditions on two new functions $c(x)$ and $g(x).$ What is more, more conditions on the processes are also expected. This leads to a theorem borrowed from \cite{wentzell-a-course}.
\begin{theorem}[Section 10.3 in \cite{wentzell-a-course}]
Let $\xi^{\epsilon}$ be uniformly stochastically continuous, the function $g(x)$ be in the domain of the generating operator $A^{\epsilon}_t,$ and the function $c(x)$ be bounded uniformly continuous. Then the function given by (\ref{eq:connections}) is the unique solution of (\ref{eq:Math-Physics-general}).
\end{theorem}

For (\ref{eq:connections}), we need further assumptions in order to apply Theorem \ref{main-thm}. For instance, it is assumed that $\max\left[\int_0^tc(\phi(s))-H_0ds\right]$ is reached uniquely at non-zero $\phi_0$.

\section{Appendix}\label{appendix}
\subsection{Compensating operators after transformations}
In preceding sections, we presented many compensating operators after transformations without any proofs. In order to show the method, we give the details for deriving the compensating operator appeared in Section \ref{subsec:finiteness-coefficient}. There, Gaussian process $\eta$ was
considered with compensating operator
\begin{align*}
\mathfrak{A}^{\eta}f(t,x)=\frac{\partial f}{\partial
t}(t,x)+x\cdot\frac{\partial^2 G_0}{\partial z
\partial x}(t,\phi_0(t);&z_0(t))\cdot\frac{\partial
f}{\partial x}(t,x)+ \frac{1}{2}\frac{\partial^2 G_0}{\partial
z^2}(t,\phi_0(t);z_0(t))\cdot\frac{\partial^2 f}{\partial x^2}(t,x).
\end{align*}
for $f(t,x)$ which is bounded and continuous together with its first
derivatives in $t$ and $x$ and its second derivative in $x.$ The
following transformation was used
$$\widetilde{\eta}_t=\exp\left\{-\int_0^tg(s)ds\right\}\eta_t,\quad t\in[0,T].$$
After such a transformation, we give the details in this section that $\widetilde{\eta}$ has compensating operator given by
\begin{align*}
\mathfrak{A}^{\widetilde{\eta}}f(t,x)&=\frac{\partial f}{\partial
t}(t,x)+\frac{1}{2}\frac{\partial^2 G_0}{\partial
z^2}(t,\phi_0(t);z_0(t))\cdot\frac{\partial^2 f}{\partial
x^2}(t,x)\exp\left\{-2\int_0^tg(s)ds\right\}
\end{align*}
for the same class of functions $f.$

The generating operator $A^{\eta}_t$ of process $\eta$ is
$A^{\eta}_tf(x)=\mathfrak{A}^{\eta}f(t,x)$ for $f(t,x)=f(x).$ Let us
assume the starting position of process $\eta$ is $\eta_s=x.$ From
definitions of compensating operator and generating operator, we
have the following two equalities:
\begin{align*}
P^{s,t}_{\eta}f(x)-f(x)&=\int_s^tP^{s,v}_{\eta}A^{\eta}_vf(x)dv,\\
P^{s,t}_{\eta}f(t,\cdot)(x)-f(s,x)&=\int_s^tP^{s,v}_{\eta}\mathfrak{A}^{\eta}f(v,\cdot)(x)dv,
\end{align*}
for suitable $f(x)$ and $f(t,x),$ where $P^{s,t}_{\eta}$ is the
multiplicative family of operators of Markov process $\eta$ given by
$P^{s,t}_{\eta}f(x)=\mathbb{E}^{\eta}_{s,x}f(\eta_t).$ The following two
connections between multiplicative families of $\eta$ and
$\widetilde{\eta}$ are easily derived:
\begin{align*}
P^{s,t}_{\eta}f(x)&=\mathbb{E}^{\eta}_{s,x}f(\eta_t)=\mathbb{E}^{\widetilde{\eta}}_{s,x\exp\{-\int_0^sg(u)du\}}f\left(\widetilde{\eta}_t\exp\{\int_0^tg(u)du\}\right)\\
&=P^{s,t}_{\widetilde{\eta}}G(t,\cdot)\left(x\exp\{-\int_0^sg(u)du\}\right),\quad G(t,x)=f\left(x\exp\{\int_0^tg(u)du\}\right),\\
P^{s,t}_{\widetilde{\eta}}f(x)&=\mathbb{E}^{\widetilde{\eta}}_{s,x}f(\widetilde{\eta}_t)=\mathbb{E}^{\eta}_{s,x\exp\{\int_0^sg(u)du\}}f\left(\eta_t\exp\{-\int_0^tg(u)du\}\right)\\
&=P^{s,t}_{\eta}F(t,\cdot)\left(x\exp\{\int_0^sg(u)du\}\right),\quad
F(t,x)=f\left(x\exp\{-\int_0^tg(u)du\}\right).
\end{align*}
Now we look for $A_t^{\widetilde{\eta}}$ in the following way.
\begin{align*}
P^{s,t}_{\widetilde{\eta}}f(x)-f(x)&=P^{s,t}_{\eta}F(t,\cdot)\left(x\exp\{\int_0^sg(u)du\}\right)-F\left(s,x\exp\{\int_0^sg(u)du\}\right)\\
&=\int_s^tP^{s,v}_{\eta}\mathfrak{A}^{\eta}F(v,\cdot)\left(x\exp\{\int_0^sg(u)du\}\right)dv\\
&=\int_s^tP^{s,v}_{\eta}g(v,\cdot)\left(x\exp\{\int_0^sg(u)du\}\right)dv,\text{ set }h(v,x)=\mathfrak{A}^{\eta}F(v,\cdot)(x)\\
&=\int_s^tP^{s,v}_{\widetilde{\eta}}\widetilde{h}(v,\cdot)(x)dv,\text{
where }\widetilde{h}(v,x)=h\left(v,x\exp\{\int_0^vg(u)du\}\right).
\end{align*}
It is straightforward to compute
$$h(v,x)=\frac{1}{2}\frac{\partial^2 G_0}{\partial
z^2}(v,\phi_0(v);z_0(v))\cdot
f''\left(x\exp\{-\int_0^vg(u)du\}\right)\cdot\exp\left\{-2\int_0^vg(u)du\right\},$$
thus
\begin{align*}
A^{\widetilde{\eta}}_tg(x)&=\widetilde{h}(t,x)=h\left(t,x\exp\{\int_0^tg(u)du\}\right)\\
&=\frac{1}{2}\frac{\partial^2 G_0}{\partial
z^2}(t,\phi_0(t);z_0(t))\cdot
f''\left(x\right)\cdot\exp\left\{-2\int_0^vg(u)du\right\}.
\end{align*}

\subsection{Proof of Lemma \ref{lemma-01}}
\begin{proof}
The same conclusion with a continuous and bounded functional $F$ was given in \cite{Wentzell-LD-Markov-Processes-1986} without a proof. For completeness, we first present the proof for continuous and bounded $F$ and then extend the argument to include the continuous functional $F$ which is only bounded above.

We take
$$\gamma=\left[F(\phi_0)-S(\phi_0)-\max_{\|(\phi_0-\phi\|\geq\delta}[F(\phi)-S(\phi)]\right]/2.$$
Here the $\max_{\|(\phi_0-\phi\|\geq\delta}[F(\phi)-S(\phi)]$ is reached at some point $\phi_1\in X.$ To see this, we note that
$$\sup_{\|(\phi_0-\phi\|\geq\delta}[F(\phi)-S(\phi)]=\sup_{\phi\in A}[F(\phi)-S(\phi)]$$
where $$A=\left\{\phi\in D_0[0,T]: \|(\phi_0-\phi\|\geq\delta\text{ and }S(\phi)\leq \sup_{x\in X} F(x)-[F(\tilde{\phi})-S(\tilde{\phi})]\right\}$$
given a fixed $\tilde{\phi}$ such that $\|\phi_0-\tilde{\phi}\|\geq\delta$ and $\left|F(\tilde{\phi})-S(\tilde{\phi})\right|<\infty.$ The compactness of $A$ implies that $\sup_{\phi\in A}[F(\phi)-S(\phi)]$ attains its maximum at some $\phi_1.$

We first assume that $F$ is bounded, thus $F$ can be split into finitely many parts as
$$F(x)\in\bigcup_{i=-k}^k[i\gamma/4,(i+1)\gamma/4].$$
It then follows
\begin{align*}
&\int_{\left\{\|\xi^{\epsilon}-\phi_0\|\geq\delta\right\}}\exp\left\{F(\xi^{\epsilon})/\epsilon\right\}d\,\mathbb{P}^{\epsilon}\\
&\leq\sum_{i=-k}^k\int_{\left\{\|\xi^{\epsilon}-\phi_0\|\geq\delta,F(\xi^{\epsilon})\in[i\gamma/4,(i+1)\gamma/4]\right\}}\exp\left\{F(\xi^{\epsilon})/\epsilon\right\}d\,\mathbb{P}^{\epsilon}\\
&\leq \sum_{i=-k}^k\exp\left\{(i+1)\gamma/(4\epsilon)\right\}\cdot \mathbb{P}^{\epsilon}\left\{\|\xi^{\epsilon}-\phi_0\|\geq\delta,F(\xi^{\epsilon})\in[i\gamma/4,(i+1)\gamma/4]\right\}.
\end{align*}
On each set $\left\{\|\xi^{\epsilon}-\phi_0\|\geq\delta,F(\xi^{\epsilon})\in[i\gamma/4,(i+1)\gamma/4]\right\},$
$$S(\xi^{\epsilon})\geq F(\xi^{\epsilon})-[F(\phi_1)-S(\phi_1)]\geq i\gamma/4-[F(\phi_1)-S(\phi_1)],$$
then according to large deviation principle, for small enough $\epsilon,$
\begin{align*}
&\mathbb{P}^{\epsilon}\left\{\|\xi^{\epsilon}-\phi_0\|\geq\delta,F(\xi^{\epsilon})\in[i\gamma/4,(i+1)\gamma/4]\right\}\\
&\leq \exp\left\{-\inf_{\left\{\phi:\left\{\|\xi^{\epsilon}-\phi_0\|\geq\delta,F(\phi)\in[i\gamma/4,(i+1)\gamma/4]\right\}\right\}}S(\phi)/\epsilon+\gamma/(4\epsilon)\right\}\\
&\leq\exp\left\{[F(\phi_1)-S(\phi_1)]/\epsilon-i\gamma/(4\epsilon)+\gamma/(4\epsilon)\right\}.
\end{align*}
Therefore,
\begin{align*}
&\int_{\left\{\|\xi^{\epsilon}-\phi_0\|\geq\delta\right\}}\exp\left\{F(\xi^{\epsilon})/\epsilon\right\}d\,\mathbb{P}^{\epsilon}\\
&\leq \sum_{i=-k}^k\exp\left\{(i+1)\gamma/(4\epsilon)\right\}\cdot \exp\left\{[F(\phi_1)-S(\phi_1)]/\epsilon-i\gamma/(4\epsilon)+\gamma/(4\epsilon)\right\}\\
&=\sum_{i=-k}^k\exp\left\{\frac{1}{\epsilon}\left[F(\phi_0)-S(\phi_0)-\frac{3\gamma}{2}\right]\right\}=o\left(\exp\left\{[F(\phi_0)-S(\phi_0)-\gamma]/\epsilon\right\}\right).
\end{align*}

Now we assume $F$ to be bounded above, i.e., $M:=\sup_{x\in D_0[0,T]} F(x)<\infty.$ Let us define a sequence of truncated functionals $G^N$ as follows
\begin{align*}
G^N(x)=\begin{cases}F(x)& \text{ if }F(x)\geq -N,\\
-N& \text{ if }F(x)< -N.
\end{cases}
\end{align*}
Taking into account the fact that $F\leq G,$ it is clear that the proof is complete if we can prove $G^N-S$ attains its maximum uniquely at $\phi_0$ for large $N$ given the condition that $F-S$ attains its maximum uniquely at $\phi_0.$ We will argue this by contradiction. Suppose for every large $N,$ there is a point $\phi_N\in X$ different from $\phi_0$ such that
\begin{align}\label{eq:near-end}
G^N(\phi_N)-S(\phi_N)\geq G^N(\phi_0)-S(\phi_0)=F(\phi_0)-S(\phi_0)\text{ for large }N.
\end{align}
Noticing that
$$\sup_{\phi\in X}[G^N(\phi)-S(\phi)]=\sup_{\phi\in B}[G^N(\phi)-S(\phi)]$$
for a compact set $B=\left\{\phi\in X: S(\phi)\leq M-[F(\phi_0)-S(\phi_0)]\right\},$ we have
$\{\phi_N\}\subseteq B.$ Then there must be a limiting point $\widehat{\phi}$ of a subsequence of $\{\phi_N\}$ (we still denote the subsequence as $\phi_N$),
$$\lim_{N\rightarrow\infty}\phi_N=\widehat{\phi}.$$
From (\ref{eq:near-end}), it follows
$$F(\phi_N)\leq -N$$
because of the uniqueness of the maximizer of $F-S.$ Now we take the limit
$$F(\widehat{\phi})=\lim_{N\rightarrow\infty}F(\phi_N)\leq \lim_{N\rightarrow\infty}-N=-\infty,$$
which is impossible.
\end{proof}

\subsection{On a variational problem}\label{subsec:variational}
In this section, we show the existence and uniqueness of the variational problem of Example 1 in Section \ref{subsec:examples} (note that $T=1$ in Example 1)
$$\max_{\phi\in C_0^1[0,T]}\int_0^T\left[\phi(t)-\phi(t)^2-\left(\phi'(t)\ln\left(\phi'(t)+\sqrt{\phi'(t)^2+1}\right)+1-\sqrt{\phi'(t)^2+1}\right)\right]dt.$$

The proof of uniqueness of our problem is standard and is included
in Section \ref{Uniqueness}. For existence, many references deal
with problems having two fixed boundaries and satisfying coercivity
assumption (see (\ref{coercivity}) in Section \ref{Existence}), our
problem fails to meet these two requirements. A proof for the
existence is given in Section \ref{Existence} mainly based on nice
properties of the functional $F(\phi)-S(\phi)$ and the analysis on
absolutely continuous function space.

\subsubsection{Uniqueness}\label{Uniqueness} For short, let us define,
\begin{align}\label{H}
&H(u)=u\ln\left(u+\sqrt{u^2+1}\right)+1-\sqrt{u^2+1},\\
&v(\phi)=\int_0^T\left[\phi(t)-\phi(t)^2-H(\phi'(t))\right]dt.
\end{align}
Let $f(x,y)=H(y)+x^2-x,$ then the variational problem becomes
\begin{align}\label{directly}
\alpha=\max_{\phi\in
C_0^1[0,T]}\int_0^T\left[\phi(t)-\phi(t)^2-H(\phi'(t))\right]dt=-\min_{\phi\in
C_0^1[0,T]}\int_0^Tf(\phi(t),\phi'(t))dt.
\end{align}
Now suppose $\phi_1$ and
$\phi_2$ are two minimizers of problem (\ref{directly}). Let
$w(t)=[\phi_1(t)+\phi_2(t)]/2,$ then on one hand,
$\int_0^Tf(w(t),w'(t))dt\geq-\alpha;$ on the other hand, convexity
of $f$ yields
\begin{align*}
\int_0^Tf(w(t),w'(t))dt&=\int_0^Tf\left(\frac{1}{2}(\phi_1(t),\phi_1'(t))+\frac{1}{2}(\phi_2(t),\phi_2'(t))\right)dt\\
&\leq\frac{1}{2}\int_0^Tf(\phi_1(t),\phi_1'(t))dt+\frac{1}{2}\int_0^Tf(\phi_2(t),\phi_2'(t))dt=-\alpha,
\end{align*}
which indicates that $w(t)$ is also a minimizer of (\ref{directly}).
From equality
\begin{align}\label{temp-zero}
\int_0^T\left[\frac{1}{2}f(\phi_1(t),\phi_1'(t))+\frac{1}{2}f(\phi_2(t),\phi_2'(t))-f(w(t),w'(t))\right]dt=-\frac{1}{2}\alpha-\frac{1}{2}\alpha+\alpha=0
\end{align}
where the integrand of (\ref{temp-zero}) is always nonpositive (from
convexity of $f$), we have
$$\frac{1}{2}f(\phi_1(t),\phi_1'(t))+\frac{1}{2}f(\phi_2(t),\phi_2'(t))=f(w(t),w'(t)),\quad\text{ for all }t\in[0,T].$$
Rewrite above identity as follows
\begin{align}\label{temp-zero-zero}
\frac{1}{2}\phi_1^2(t)+\frac{1}{2}\phi_2^2(t)-\left(\frac{\phi_1(t)+\phi_2(t)}{2}\right)^2&=H(\frac{\phi_1'(t)+\phi_2'(t)}{2})-\left(\frac{1}{2}H(\phi_1'(t))+\frac{1}{2}H(\phi_2'(t))\right).
\end{align}
If there were a point $t_0\in[0,T]$ such that
$\phi_1(t_0)\neq\phi_2(t_0),$ then left hand side of
(\ref{temp-zero-zero}) would be strictly $<0$ (which is from strict
convexity of function $x^2$), while the right hand side is always
$\geq0$ from convexity of $H.$ This contradiction proves that
(\ref{directly}) has at most one minimizer.

\subsubsection{Existence}\label{Existence} When one deals with the existence of variational
problems, the following coercivity condition is in general assumed:
for all $p,z\in R,$
\begin{align}\label{coercivity}
H(p)-(z-z^2)\geq\alpha|p|^q-\beta,\qquad\exists\,\,\alpha>0,\beta\geq0,q>1.
\end{align}
(see Section 8.2 in \cite{Evans}, or see \cite{Tonelli} for the case
$q=2$). But this condition is not satisfied for our problem since
$\lim_{|p|\rightarrow\infty}\frac{H(p)-(z-z^2)}{|p|^q}=0$ for any
fixed $z.$ What is more, calculus of variations in references were
given in general with two fixed boundaries: $\phi(0)=A$ and
$\phi(T)=B.$ But our problem has one movable boundary $\phi(T).$

We define a space $\text{AC}_0[0,T]$ consisting of absolutely
continuous functions on $[0,T]$ vanishing at zero:
$$\text{AC}_0[0,T]=\left\{f:[0,T]\rightarrow \mathbb{R}\text{ being absolutely continuous with }f(0)=0\right\}.$$
The existence of our variational problem is solved in the following
way. We first prove the existence for $\max_{\phi\in
\text{AC}_0[0,T]}v(\phi),$ which implies that this variational problem coincides with a two fixed boundary problem
$$\mathop{\max_{\phi\in \text{AC}_0[0,T]}^{
}}_{\phi(T)=c}v(\phi), \text{ for some }c.$$
Then we show $C^1$ regularity of the
maximizer by means of two fixed boundary variational results,
which immediately implies the existence of $\max_{\phi\in
C_0^1[0,T]}v(\phi).$

\subsubsection*{Existence of $\max_{\phi\in
\text{AC}_0[0,T]}v(\phi)$}\label{existence-in-c-0-0} Obviously, we
can find some $\phi_{*}\in \text{AC}_0[0,T]$ with
$|v(\phi_{*})|<T/4$ (for instance $\phi_{*}\equiv0$). We define a
subset $\mathcal{A}$ of $\text{AC}_0[0,T],$
$$\mathcal{A}=\left\{\phi\in \text{AC}_0[0,T]:\int_0^TH(\phi'(t))dt\leq\frac{T}{4}-v(\phi_{*})\right\}.$$
Then
\begin{align}\label{compact-set}
\sup_{\phi\in \text{AC}_0[0,T]}v(\phi)=\sup_{\phi\in
\mathcal{A}}v(\phi).
\end{align}
To see (\ref{compact-set}), we notice that for any $\phi\notin
\mathcal{A},$
$$\int_0^TH(\phi'(t))dt>\frac{T}{4}-v(\phi_{*}),\text{ then }v(\phi_{*})>\frac{T}{4}-\int_0^TH(\phi'(t))dt\geq v(\phi).$$
Let us write
$$\alpha=\sup_{\phi\in \mathcal{A}}v(\phi),$$
and let $\{\phi_n(t)\}_{n\geq1}\subseteq\mathcal{A}$ be chosen such
that
\begin{align*}
\lim_{n\rightarrow\infty}v(\phi_n)=\alpha,\text{ and
}\lim_{n\rightarrow\infty}\max_{0\leq t\leq
T}|\phi_n(t)-\phi_0(t)|=0\text{ for some }\phi_0\in
\text{AC}_0[0,T].
\end{align*}
The reason why we can choose such a sequence $\phi_n$ is from the
fact that $\mathcal{A}$ is compact in $\text{AC}_0[0,T]$ according
to the following Lemma \ref{lower-semi-con-compact} (after passing
to a subsequence). We now show $v(\phi_0)=\alpha.$ In fact,
$v(\phi_0)\leq\alpha$ is trivial. Lower semi-continuity of
$-v(\cdot)$ in Lemma \ref{lower-semi-con-compact} gives
$$v(\phi_0)\geq\limsup_{n\rightarrow\infty}v(\phi_n)=\alpha.$$
So (\ref{compact-set}) can be rewritten as
\begin{align*}
\max_{\phi\in \text{AC}_0[0,T]}v(\phi)=\max_{\phi\in
\mathcal{A}}v(\phi),
\end{align*}
which proves the existence of $\max_{\phi\in
\text{AC}_0[0,T]}v(\phi).$

\begin{lemma}\label{lower-semi-con-compact}
$\mathcal{A}$ defined above is compact in $\text{AC}_0[0,T]$ and
$-v(\phi)$ is lower semi-continuous in $\text{AC}_0[0,T]$ in uniform
topology.
\end{lemma}
\begin{proof} We will
finish the proof in several steps. First we show $\mathcal{A}$ is an
absolutely euqicontinuous family of functions: for any $\epsilon>0,$
there is $\delta(\epsilon)>0$ such that whenever finitely many
non-overlapping intervals $\sum_i(t_i-s_i)\leq \delta,$ then
\begin{align}\label{equicontinuous}
\sum_i|\phi(t_i)-\phi(s_i)|<\epsilon,\qquad\forall\,\,\phi\in\mathcal{A}.
\end{align}
To see (\ref{equicontinuous}), first we have a nice property for
$H:$
\begin{align}\label{nice-property-H}
H(p)\geq|p|\cdot\ln\left(|p|+\sqrt{p^2+1}\right)+1-\sqrt{2}|p|-\sqrt{2},\qquad\text{
for all }p\in \mathbb{R}.
\end{align}
Then $\lim_{|p|\rightarrow\infty}H(p)/|p|=\infty,$ so there is some
$P(\epsilon)>0,$ such that when $|p|>P,$
$$H(p)/|p|\geq2\left[\frac{1}{4}-v(\phi_{*})\right]/{\epsilon}.$$
Let $\delta(\epsilon)=\epsilon/(2P),$ then for all
$\phi\in\mathcal{A},$
\begin{align*}
\frac{1}{4}-v(\phi_{*})&\geq\int_0^TH(\phi'(t))dt\geq\sum_i\int_{s_i}^{t_i}H(\phi'(t))dt\geq\sum_i\int_{s_i}^{t_i}\frac{H(\phi'(t))}{|\phi'(t)|}|\phi'(t)|1_{\{|\phi'(t)|>P\}}(t)dt\\
&\geq\sum_i\int_{s_i}^{t_i}|\phi'(t)|1_{\{|\phi'(t)|>P\}}(t)dt\cdot2\left[\frac{1}{4}-v(\phi_{*})\right]/{\epsilon},
\end{align*}
so
\begin{align*}
&\epsilon/2\geq\sum_i\int_{s_i}^{t_i}|\phi'(t)|1_{\{|\phi'(t)|>P\}}(t)dt=\sum_i\int_{s_i}^{t_i}|\phi'(t)|dt-\sum_i\int_{s_i}^{t_i}|\phi'(t)|1_{\{|\phi'(t)|\leq
P\}}(t)dt,\\
\Rightarrow\,\,\,&\sum_i|\phi(t_i)-\phi(s_i)|\leq\sum_i\int_{s_i}^{t_i}|\phi'(t)|dt\leq\epsilon/2+\sum_i\int_{s_i}^{t_i}|\phi'(t)|1_{\{|\phi'(t)|\leq
P\}}(t)dt\\
&\leq\epsilon/2+\epsilon/2=\epsilon.
\end{align*}

The second step is to prove lower semi-continuity of $-v(\cdot).$
Let $\phi_n\in \text{AC}_0[0,T]$ be a family of absolutely
continuous functions such that $\max_{0\leq t\leq
T}|\phi_n(t)-\phi_{\infty}(t)|\rightarrow0$ as $n\rightarrow\infty.$
It turns out that $\phi_{\infty}$ is also absolutely continuous.
More precisely, according to absolute equicontinuity
(\ref{equicontinuous}), for any $\epsilon>0,$  there is
$\delta(\epsilon)>0$ such that if $\sum_i(t_i-s_i)<\delta,$ then
$\sup_n\sum_i|\phi_n(t_i)-\phi_n(s_i)|<\epsilon.$ Sending
$n\rightarrow\infty$ we get
$\sum_i|\phi_{\infty}(t_i)-\phi_{\infty}(s_i)|<\epsilon,$ which
proves the absolute continuity of $\phi_{\infty}.$ Now we show the
lower semi-continuity of $\int_0^TH(\phi(t))dt.$ Let
$0=t_0<t_1\cdots<t_k=T,$ Jensen's inequality implies
\begin{align*}
&\liminf_{n\rightarrow\infty}\int_0^TH(\phi_n'(t))dt=\liminf_{n\rightarrow\infty}\sum_{i=0}^{k-1}\int_{t_i}^{t_{i+1}}H(\phi_n'(t))dt\\
&\geq\liminf_{n\rightarrow\infty}\sum_{i=0}^{k-1}(t_{i+1}-t_i)H\left(\frac{\phi_n(t_{i+1})-\phi_n(t_{i})}{t_{i+1}-t_{i}}\right)=\sum_{i=0}^{k-1}(t_{i+1}-t_i)H\left(\frac{\phi_{\infty}(t_{i+1})-\phi_{\infty}(t_{i})}{t_{i+1}-t_{i}}\right)\\
&=\sum_{i=0}^{k-1}\int_{t_i}^{t_{i+1}}H\left(\Psi(t)\right)dt,\quad\text{where }\Psi(t)=\frac{\phi_{\infty}(t_{i+1})-\phi_{\infty}(t_{i})}{t_{i+1}-t_{i}}\text{ for }t_i\leq t<t_{i+1}\\
&=\int_{0}^{T}H\left(\Psi(t)\right)dt.
\end{align*}
Now let a sequence $\triangle_m$ of partitions be infinitely small,
then the corresponding functions $\Psi_m(t)$ converge to
$\phi_{\infty}'(t)$ almost everywhere (because of absolute
continuity of $\phi_{\infty}$). Using continuity of $H$ and Fatou's
lemma we get
$$\int_{0}^{T}H\left(\phi_{\infty}'(t)\right)dt\leq\liminf_{n\rightarrow\infty}\int_0^TH(\phi_n'(t))dt,$$
which gives us the lower semi-continuity of
$\int_{0}^{T}H(\phi(t))dt.$ Then the lower semi-continuity of
$-v(\cdot)$ is from lower semi-continuity of
$\int_{0}^{T}H(\phi(t))dt.$

The last step will present the compactness of $\mathcal{A}$ in
$\text{AC}_0[0,T].$ Lower semi-continuity of
$\int_{0}^{T}H(\phi(t))dt$ shows $\mathcal{A}$ is closed in
$\text{AC}_0[0,T].$ What's more, the equicontinuity in step one and
the fact all functions in $\mathcal{A}$ have zero initial value
imply that $\mathcal{A}$ is pre-compact in $C_0[0,T],$ thus
$\mathcal{A}$ is compact in $\text{AC}_0[0,T].$
\end{proof}

\subsubsection*{$C^1$ regularity of a maximizer of $\max_{\phi\in
\text{AC}_0[0,T]}v(\phi)$}\label{regularity-in-c-1}
Let us consider two fixed boundaries problem as follows
\begin{align*}
g(c):=\mathop{\max_{\phi\in \text{AC}_0[0,T]}^{
}}_{\phi(T)=c}v(\phi).
\end{align*}
First we note that $g(c)$ is well defined because of the existence
of a maximizer of $v(\phi)$ under restrictions $\phi\in
\text{AC}_0[0,T]$ and $\phi(T)=c.$ Clarke and Vinter in their paper
\cite{Clarke-Vinter} showed several powerful regularity theorems
under pretty mild hypotheses by using nonsmooth analysis.

More precisely, Clarke and Vinter in \cite{Clarke-Vinter} considered the basic problem in the calculus of variation, which is to minimize
$$J(\phi):=\int_0^T-L(\phi(t),\phi'(t))dt$$
over the class of absolutely continuous functions $\phi$ having two fixed boundaries $\phi(0)=A$ and $\phi(T)=B.$ Assuming $L$ satisfies suitable conditions, they proved that if there is a $\phi(t)$ solving the variational problem, then at every point $t\in[0,T],$ the function $\phi$ is $C^{\infty}$ in a neighborhood of $t,$ see Theorem 2.1 and Corollary 3.1 in \cite{Clarke-Vinter}. Our functional
$\phi(t)-\phi(t)^2-H(\phi'(t))$ has nice properties which make it
satisfy all the hypotheses in \cite{Clarke-Vinter}, so we immediately deduce the $C^1$
regularity (actually $C^{\infty}$ regularity) of the maximizer of $g(c).$

Now, according to Section \ref{existence-in-c-0-0},
\begin{align*}
\max_{\phi\in \text{AC}_0[0,T]}v(\phi)=\mathop{\max_{\phi\in
\text{AC}_0[0,T]}^{ }}_{\phi(T)=c}v(\phi)\,\,\text{ for some
(possibly not unique) }c\in \mathbb{R}.
\end{align*}
because of the existence of $\max_{\phi\in \text{AC}_0[0,T]}v(\phi).$

\section*{Acknowledgment}
The author wishes to thank Professor Alexander Wentzell for his guidance and suggestions on this work, and Portuguese Science Foundation project (PTDC/MAT/120354/2010) for the support.

\end{document}